\documentclass{article}
\usepackage{stmaryrd}

\usepackage[utf8]{inputenc}
\usepackage[T1]{fontenc}
\usepackage[british,UKenglish,USenglish,english,american]{babel}

\usepackage{hyperref}
\usepackage{amsmath}
\usepackage{amssymb}
\usepackage{amsthm}
\usepackage{mathrsfs}
\usepackage{amsfonts}

\usepackage{esint}
\usepackage{array}
\usepackage{latexsym}

\usepackage{lmodern}

\usepackage{graphicx}
\usepackage{graphics}
\usepackage{epsfig}
\usepackage{color}
\DeclareGraphicsExtensions{.eps,.pdf,.jpeg,.png}

\usepackage{titlesec}
\usepackage{multirow}

\usepackage[all]{xy}
\usepackage{dsfont}

\usepackage{hyperref}
\usepackage{bigints}

\def\ind{\operatorname{ind}}

\def\Tr{\operatorname{Tr}}

\def\Z{\mathbb Z}

{\theoremstyle {definition} \newtheorem {defi} {Definition} [section] }
{\theoremstyle {plain}  \newtheorem {thm} [defi] {Theorem}}

{\theoremstyle {plain}  \newtheorem {cor} [defi]{Corollary}}

{\theoremstyle {plain} \newtheorem {prop} [defi]{Proposition}}

{\theoremstyle {plain} \newtheorem {lem}[defi] {Lemma}}

{\theoremstyle {definition} }

{\theoremstyle {definition} \newtheorem{remarque}[defi]{Remark}}
{\theoremstyle {definition} }
{\theoremstyle {definition} }
{\theoremstyle {definition}  }
{\theoremstyle {definition} }

\def\Tr{{\mathrm{Tr}}}

\def\g{{\mathfrak{g}}}

\def\K{{\mathrm{K_G}}}
\def\Kh{{\mathrm{K_{G\times H}}}}
\def\k{{\mathrm{K}}}
\def\ind{{\mathrm{Ind^{M|B}}}}

\def\indm{{\mathrm{Ind^{M|B}_G}}}
\def\indmsurh{{\mathrm{Ind^{M/H|B}_G}}}

\def\ext{{\mathrm{ext}}}
\def\coker{{\mathrm{coker}}}
\def\interieur{{\mathrm{int}}}

\usepackage{color}
\definecolor{darkgreen}{cmyk}{1,0,1,.2}
\definecolor{m}{rgb}{1,0.1,1}
\definecolor{green}{cmyk}{1,0,1,0}
\definecolor{darkred}{rgb}{0.55, 0.0, 0.0}
\definecolor{test}{rgb}{1,0,0}
\definecolor{cmyk}{cmyk}{0,1,1,0}

\usepackage{authblk}
\newcommand{\email}[1]{\href{mailto:#1}{#1}}

\title{The index of $G$-transversally elliptic families I}
\author[1]{Alexandre Baldare}
\affil[1]{Institut \'Elie Cartan de Lorraine, Université de Lorraine, 57070 Metz, France, \email{alexandre.baldare@univ-lorraine.fr}}

\begin{document}

\maketitle

\begin{abstract}
We define and study the index map for families of $G$-transversally elliptic operators and introduce the multiplicity for a given irreducible representation as a virtual bundle over the base of the fibration.  We then prove the usual axiomatic properties for the index map extending the Atiyah-Singer results \cite{atiyah1974elliptic}. Finally, we compute the Kasparov intersection product of our index class against the $\k$-homology class of an elliptic operator on the base.  Our approach is based on the functorial properties of the intersection product, and relies on some constructions due to Connes-Skandalis   and to Hilsum-Skandalis. 

\end{abstract}

\tableofcontents

%-------------------------------------------------------------------------------------------------------
\section{Introduction}

The study of transversally elliptic operators was initiated in the seminal work of Atiyah-Singer, see  \cite{atiyah1974elliptic}. In this latter paper, although the equivariant Fredholm index is not well defined in general for such operators, Atiyah and Singer  proved that any irreducible representation appears with a finite multiplicity and introduced an important index invariant, which is now a slowly increasing distribution on the given compact Lie group that they called the {\em{distributional index}}. They also proved many axiomatic properties for this distributional index, which allowed them to reduce its computation to the case of actions of tori on euclidian spaces and therefore to fully compute it in many interesting cases including locally free actions of general compact Lie groups, and general circle actions. Although an important progress has been achieved during the last decades, the problem of computing the Atiyah-Singer distributional index for transversally elliptic operators, or equivalently the multiplicities of the irreducible representations, is  still open. The goal of this paper is to extend Atiyah-Singer results to families of transversally elliptic operators using the powerful machinery of Kasparov's bivariant K-theory through the deep  associativity of the  intersection product \cite{Kasparov:KKtheory}.

\medskip

Let $G$ be a compact Lie group and let $M$ be a compact $G$-manifold. 
A $G$-transversally elliptic pseudodifferential operator on $M$ is a $G$-invariant pseudodifferential operator whose principal symbol becomes invertible when it is restricted to (non-zero) orbit transverse covectors of the $G$-action. For this class of operators, Atiyah and Singer defined an index class $\mathrm{Ind}(P)$ living in the space $C^{-\infty}(G)^{Ad(G)}$ of $Ad$-invariant distributions on $G$, i.e.
$$
\mathrm{Ind}(P) \in C^{-\infty}(G)^{Ad(G)}.
$$
More precisely, for a $G$-transversally elliptic operator $P$, the evaluation of the index class against a smooth function $\varphi \in C^{\infty}(G)$ is defined by the formula (see \cite{atiyah1974elliptic}) 
$$
\mathrm{Ind}(P)(\varphi)= \int_G \varphi(g) \big(\Tr_{|\ker(P)}(g) - \Tr_{|\coker(P)}(g)\big) dg,
$$
and induces the so-called Atiyah-Singer index map
$$
\mathrm{Ind} : \K(T^*_GM) \longrightarrow C^{-\infty}(G)^{Ad(G)}.
$$ 
%whose properties, already highlighted in \cite{Atiyah-Singer:I}, allow Atiyah-Singer to reduce its computation to the case of a Euclidean space equipped with a torus action.

In \cite{julg1982induction} Julg associated with any $G$-transversally elliptic operator $P$, a K-homolo\-gy class $[P]\in\k\k(C(M)\rtimes G ,\mathbb{C})$ whose restriction class in $\k\k(C^*G , \mathbb{C})$ is an index class closely related to the Atiyah-Singer distribution through the Connes-Chern character. In \cite{BV:IndEquiTransversal}, Berline and Vergne proved an index theorem for $G$-transversally elliptic operators in the context of equivariant cohomology. More recently, Kasparov has extended all known K-theory constructions for transversally elliptic operators to a larger class of proper actions on complete manifolds. In this context, although the index class is not defined, the Julg class $[P]$ still makes sense and plays the role of the index class, see \cite{Kasparov:KKindex}. It is worth pointing out that Kasparov even proved many index theorems as equalities in his bivariant $\k$-theory.

\medskip

In this paper we investigate the extension of classical Atiyah-Singer results to families of transversely elliptic operators and restrict ourselves to the case of compact Lie groups actions on compact fibrations. Notice that  the  case of equivariant elliptic famillies and even equivariant longitudinal elliptic operators on foliations are nowadays classical,  see for instance \cite{Benameur:LongLefschetzKtheorie, Benameur:thmFamilleLefschetz, Benameur:flat:bundles}.
As in the case of non-equivariant elliptic families \cite{Atiyah-Singer:IV}, we mainly allow the base of our fibration to be a compact space, as the smoothness condition is only requiered to compute the Kasparov's intersection product with an elliptic operator on $B$. Let $p : M \rightarrow B$ be a given compact locally trivial $G$-fibration and $T^VM=\ker p_*$ be the vertical tangent bundle.
Using a $G$-invariant metric on $M$, we identify the cotangent bundle $T^*M$ with the tangent bundle $TM$.
Assuming that $G$ acts trivially on the base $B$, a family of pseudodifferential operators $A=(A_b)_{b\in B}$ in the sense of \cite{Atiyah-Singer:IV} will be $G$-transversally
elliptic if its restrictions to the fibers are $G$-transversally elliptic pseudodifferential operators. In other words, we assume that the restriction to non-zero covectors of the principal symbol $\sigma (A)$ of the family $A$ is invertible  in $T^V_GM : =T^VM \cap T_G^*M$, using standard notations, see the next section.
Following \cite{julg1982induction} and \cite{Kasparov:KKindex}, we associate with such family of $G$-transversely elliptic operators an index class, now living in the bivariant Kasparov group $\k\k(C^*G,C(B))$.

\begin{defi}
Given a $G$-invariant family $P_0 : C^{\infty{,}0}(M{,}E^+) \rightarrow C^{\infty {,} 0}(M{,}E^-)$ of $G$-transversally elliptic operators, the index class $\indm (P_0)$ of $P_0$  is defined as the $\k\k$-class:
$$
\indm (P_0) \; :=\; [\mathcal{E},\pi,P]\;\; \in \k\K(C^*G,C(B)),
$$
where $\mathcal{E}$ is the  Hilbert $C(B)$-module associated with  the hermitian $\Z_2$-graded bundle $E$,  $\pi$ is the representation of $C^*G$ on $\mathcal{E}$ induced by the $G$-action  on $E$, and $P=\begin{pmatrix} 0& P_0^*\\P_0&0
\end{pmatrix}$.
We denote by  $\ind(P_0)$ the image of this class in $\k\k(C^*G,C(B))$. 
\end{defi}

The index class only depends on the homotopy class of the principal symbol and  induces the following index map
$$
\ind : \k(T^V_GM) \longrightarrow \k\k(C^*G ,C(B)).
$$

Finally, the notion of K-theory multiplicity can be interpreted as a virtual bundle over the base $B$, generalizing the standard multiplicity for $G$-transversally elliptic operators \cite{atiyah1974elliptic}.\\

Following \cite{atiyah1974elliptic}, this paper is mostly devoted to the investigation of basic properties of our index map, i.e. free actions, multiplicativity, excision and also induction and behaviour with respect to topological shrieck maps.
These properties allow us to reduce the  computation of the index map to the case of a trivial fibration $B \times V \rightarrow B$, where $V$ is a Euclidean space equipped with a torus action, giving the full generalization of the Atiyah-Singer work.
For instance, the compatibility theorem can be stated as follows (see section \ref{NaturalityInduction} for the details):

\begin{thm}$\cite{atiyah1974elliptic}$\label{thm:naturalité:!}
Let $j :M\hookrightarrow M'$ be a $G$-embedding over $B$ with $M$ compact.
Then the following diagram is commutative:
$$\xymatrix{\K(T^V_GM) \ar[r]^{j_!} \ar[d]_{\ind}& \K(T^V_GM') \ar[d]^{\mathrm{Ind}^{\mathrm{M'|B}}} \\
\k\k(C^*G ,C(B)) \ar@{=}[r] &\k\k(C^*G ,C(B)).
}$$
\end{thm}

For a given elliptic operator $Q$ on the base $B$ and using results from \cite{hilsum1987morphismes}, we compute the Kaspavov product of our index class with the $\k$-homology class of $Q$ by exploiting the convenient description of the index class through the morphism induced from the representation ring $R(G)$ to the topological $\k$-theory $\k(B)$. This latter intersection product computation is crucial to define a distributional index à la Atiyah for families of $G$-transversally elliptic operators, and to prove a Berline-Vergne formula for families. 
These cohomological results are postponed to the forthcoming paper \cite{baldare:cohomologie}.

\medskip

This paper is organized as follows. In Section 2, we define the index class for families of $G$-transversally elliptic operators and
to introduce the $\k$-multiplicity. In Section 3 we define the index mapping and investigate its basic properties. More precisely, we formulate exact computations for free $G$-actions, prove the expected multiplicativity theorem as well as the excision theorem which are now equalities in the appropriate $\k\k$-groups.
Induction for closed subgroups and especially for the maximal torus is investigated together with the compatibility with shrieck maps.
In Section 4 and 5, we study the index class as a morphism between $R(G)$ and $\k(B)$ and compute the Kasparov product with an elliptic operator on the base, preparing the equivariant cohomological computations of our Berline-Vergne theorem.

 \medskip

\medskip
\noindent
\textbf{Acknowledgements.}
This work is part of my PhD thesis under the supervision of M.-T. Benameur. I would like to thank my advisor for very helpful discussions, comments and corrections. 
I also thank W. Liu, P.-E. Paradan and V. Zenobi for several conversations during the preparation of this work. I am also indebted to M. Hilsum, P. Piazza, M. Puschnigg and G. Skandalis for reading the PhD version of this work and for their constructive suggestions. I would also like to thank
the referee for several very useful suggestions. Last but not least, I would like to thank P. Carrillo Rouse for his interest, insight, and useful discussions during the \textit{Workshop on Index Theory, Interactions and Applications} in Toulouse.

\section{The index class of a family of $G$-transversally elliptic operators}
In this section we define the index class of a family of $G$-transversally elliptic operators and we introduce the $\k$-multiplicity of a unitary irreducible representation in the index class, a virtual vector bunde over the base $B$.

\subsection{Some preliminary results}\label{C(B):module}
This first paragraph  is devoted to a brief review of some standard results.
For most of the classical properties of  Hilbert $C^*$-modules and regular operators between them that we use here, we refer the reader to  \cite{lance1995hilbert} and \cite{C*algebre:C*module}.
The constructions given below extend the standard ones, see for instance  \cite{fox1994index}, \cite{hilsum2010bordism} and \cite{julg1988indice}.
Our hermitian scalar products will always be linear in the second variable and anti-linear in the first.
Let $G$ be a compact group and denote by $C^*G$ the $C^*$-algebra associated with $G$. Recall that any irreducible unitary representation of  $G$ naturally identifies with a finitely generated projective module on  $C^*G$ \cite{julg1981produit:croise}.

Let $p:M\rightarrow B$ be a $G$-equivariant fibration of compact manifolds, with typical fiber $F$, a compact manifold. We denote by $M_b=p^{-1}(b)$ the fiber over $b \in B$, by $T^VM = \textrm{ker}~ p_*$ the vertical subbundle of $TM$, and by $T^VM^*$ its dual bundle. We choose a $G$-invariant riemannian metric on $M$ and hence will identify  $T^VM^*$ with a  subbundle of $T^*M$ when needed.

Let $\pi : E=E^+\oplus E^- \rightarrow M$ be a $\Z_2$-graded vector bundle on $M$ which is assumed to be $G$-equivariant with a fixed $G$-invariant hermitian structure. Denote by $\mathcal{P}^m(M,E^+,E^-)$ the space of (classical) continuous families of pseudodifferential operators on $M$ as defined in \cite{Atiyah-Singer:IV}. A family $P\in \mathcal{P}^m(M,E^+,E^-)$ is $G$-equivariant if $g\cdot P=g\circ P \circ g^{-1}=P$ for any $g\in G$. 
As usual, $C^{\infty , 0}(M, E)$ will be  the space of continuous fiberwise smooth sections of $E$ over $M$, see \cite{Atiyah-Singer:IV}.

We fix from now on a $G$-invariant continuous family of Borel measures $(\mu_b)$ which are of Lebesgue class, constructed using a partition of unity of $B$. So,  
for any $f\in C(M)$, the map $b\rightarrow \int_{M_b}f(m)d\mu_b(m)$ is continuous, and each measure $\mu_b$ is fully supported in the fiber $M_b$. 
Since  $E$ is equipped with a hermitian structure, the $C(B)$-modules $C^{\infty ,0 }(M,E^\pm)$ of continuous fiberwise smooth sections over $M$, are naturally equipped with the structure of  pre-Hilbert right $G$-equivariant $C(B)$-modules with the inner product given by:
$$
\langle s,s'\rangle (b) =\int_{M_b}\langle s(m),s'(m)\rangle_{E^{\pm}_m}d\mu_b(m), \quad \text{ for } s, s'\in C^{\infty ,0 }(M,E^\pm).
$$
We denote by $\mathcal{E}^{\pm}$ the completion of $C^{\infty , 0}(M,E^\pm)$ with respect to this Hilbert structure. So, $ \mathcal{E} = \mathcal{E}^+\oplus  \mathcal{E}^-$ is our $G$-equivariant $\Z_2$-graded Hilbert  module on $C(B)$.

\begin{prop}\label{prop:pi*rep}
We define an involutive representation $\pi $ of $C^*G$ in the Hilbert $G$-module $\mathcal{E}$ by setting: 
$$
(\pi(\varphi)s)(m):=\int_G\varphi(g)(g\cdot s)(m)dg, \quad \forall s\in \mathcal{E},  \varphi \in L^1(G) .
$$\label{remarque:piGequi}\noindent
Moreover, if we endow $C^*G$ with the conjugation $G$-action, then the representation $\pi$ is $G$-equivariant.
\end{prop}

\begin{proof}
Let $s,s' \in \mathcal{E}$ and $\varphi \in L^1(G)$. By using the $G$-invariance of the metric on $E$ together with a Cauchy-Schwartz argument, we deduce that
$$
\|\langle s',\pi(\varphi )s\rangle \|_{C(B)}\leq \int_G|\varphi(g)|\|s\|\|s'\|dg\leq \|\varphi\|_{L^1(G)}\|s\|_{\mathcal E}\|s'\|_{\mathcal E}.
$$
Recall that the action of $G$ on $C(B)$ is trivial here. We then easily see that $\pi(\varphi )$ is an adjointable operator. 
Again by the $G$-invariance of the hermitian metric on $E$,  and using the $G$-invariance of the family of measures $(\mu_b)_{b\in B}$, we also get that $\pi(\varphi)^*=\pi(\varphi^*)$.
Finally, the relation $\pi(\varphi \star \psi)=\pi(\varphi)\circ \pi(\psi)$ is verified as follows: 
\begin{eqnarray*}
\pi(\varphi \star \psi)s&=&\int_G(\varphi \star \psi)(g)g\cdot s ~dg\\
&=&\int_G\int_G\varphi(h)\psi(h^{-1}g)g\cdot s ~dgdh.
\end{eqnarray*}
So setting $ h^{-1}g=k $, we get $\pi(\varphi \star \psi) = \pi(\varphi) \circ \pi(\psi)$.\\
\noindent
On the other hand, we also have $\pi(g\cdot \varphi )s=\int_G \varphi (g^{-1}hg)h\cdot s dh$ and 
setting $k=g^{-1}hg$ we get
$$
\pi(g\cdot \varphi )s=\int_G \varphi (k)gkg^{-1}\cdot s dk=(g\cdot \pi (\varphi))s. 
$$
Moreover, we have:
$$
(g\cdot \varphi g\cdot \psi )(h)=\int_G (g\cdot \varphi) (k) (g \cdot \psi )(k^{-1}h)dk
=\int_G  \varphi (g^{-1}kg) \psi (g^{-1}k^{-1}hg)dk.
$$
Finally, setting  $g^{-1}hg=t$, we obtain:
$$
(g\cdot \varphi g\cdot \psi )(h)=\int_G  \varphi (t) \psi (t^{-1}g^{-1}hg)dt=(g\cdot (\varphi \star \psi ))(h).
$$ 
By density of the range of $L^1(G)$ in $C^*G$, we get the result.
\end{proof}

\begin{remarque}
The Hilbert $C(B)$-module $\mathcal{E}$  is  associated with the continuous field of Hilbert spaces $(\prod\limits_{b\in B}L^2(M_b,E_b),\Delta)$, where $\Delta $ is the space of continuous sections of $E$, see \cite{dixmier1963champs}. 
\end{remarque}

Since $C^{\infty ,0}(M,E^{+})^{\bot}=\{0\}$, any pseudodifferential family $Q$ as before is densely defined and admits a  formal adjoint with respect to the inner product of ${\mathcal E}$ which will be denoted $Q^*$.  The domain  of $Q^*$ is
$$ 
{\rm{Dom}}(Q^*)=\{s \in \mathcal{E}^-/ \exists s'\in \mathcal{E}^+~\mathrm{such~that}~ \langle Qs'',s\rangle =\langle s'',s'\rangle ,~s''\in Dom(Q) \}.
$$
Recall that $Q^*$ is defined by its graph $G(Q^*)$ given by
$$
G(Q^*)=\{(s',s)\in \mathcal{E}^-\times \mathcal{E}^+,\forall s''\in C^{\infty ,0}(M,E^{+}), \langle s,s''\rangle =\langle s',Qs''\rangle \}.
$$
If $Q$ has pseudodifferential order $m$ and principal  symbol $q$, then $Q^*$ is also a continuous family of pseudodifferential operators of order $m$ with principal symbol $q^*_b(m,\xi)$. If $Q$ is $G$-equivariant, then so is $Q^*$.

%----------------------------------------------------------------------------

\subsection{The index class}

Assume now that $G$ is a compact {\em Lie} group with Lie algebra $\g$, and that the action of $G$ on the base $B$ is trivial.
% See Section \ref{ActionBase} for the general case.
We start by extending some results from \cite[Section 6]{Kasparov:KKindex} and use the notations from there. For $m\in M$, we denote by $f_m : G \rightarrow M$ the map given by  $f_m(g)=g\cdot m$,  by $f^{'}_m : \g \rightarrow T_mM$ its tangent map at the neutral element of $G$ and  by $f_m^{'*} : T^*_mM \rightarrow \g^*$  the dual map.
So, any $X\in \g$ defines the vector field $X^*$ given by $X^*_m:=f^{'}_m (X)$ which,  under our assumptions, is a vertical vector field, i.e. it belongs to $T^VM$. Notice also that 
$g\cdot f_m^{'}(X)=f^{'}_{g\cdot m}(\text{Ad}(g)X)$, for any $g\in G$, $m\in M$ and $X\in \g$ \cite{Kasparov:KKindex}. 
\noindent 
Let $\g_M:=M\times \g$ be the $G$-equivariant trivial bundle of Lie algebras on $M$, associated to $\g$ for the action $g\cdot (m,v)=(g\cdot m,\text{Ad}(g)v)$. 
We  endow $\g_M$ with a $G$-invariant metric and we will denote by $\|\cdot \|_m$ the associated family of Euclidean norms. The map $f^{'} :\g_M \rightarrow TM$ defined by $f^{'}(m,v)=f^{'}_m(v)$ is a $G$-equivariant vector bundle morphism. Up to normalization, we can always assume that $\forall v\in \g$, $\|f^{'}_m(v)\|\leq \|v\|_m$.  Here $\|f^{'}_m(v)\|$ is the norm given by the riemannian metric at $m$. We thus assume from now on that  $\|f^{'}_m\|\leq 1$, $\forall m\in M$. These metrics on $\g_M$ and $TM$ are also used to identify $\g_M$ with $\g_M^*$ and $TM$ with $T^*M$.  Then we can define the map $\phi : T^*M \rightarrow T^*M$ by setting $\phi_m=f^{'}_mf^{'*}_m$. Again according to the Kasparov's notations \cite{Kasparov:KKindex}, we introduce the quadratic form $q=(q_m)_{m\in M}$ on the fibers of $T^*M$ by setting: 
$$
q_m(\xi)=|\langle f^{'}_mf^{'*}_m(\xi),\xi\rangle |=\|f^{'*}_m(\xi)\|_m^2,~\forall (m,\xi)\in T^*M.
$$
We have $q_m(\xi)\leq \|\xi\|^2$, $\forall (m,\xi)\in T^*M$.
If  $\xi \in T^*_mM$, then it is easy to see that  $ \xi $ is orthogonal to the $G$-orbit   of $m$ if and only if $q_m (\xi) = 0$.

\begin{defi}
Denote by $T_GM:=\{(x,\xi)\in TM /q_m(\xi)=0\}$ the closed subspace of $TM$. We define 
$$
T^V_GM := T_GM\cap T^VM.
$$
\end{defi}

Following \cite{Kasparov:KKindex} which extends the standard definitions, see \cite{atiyah1974elliptic} and \cite{BV:ChernCharacterTransversally}, we state two definitions of families of $G$-transversally elliptic operators,  the naive definition and the technical definition.

\begin{defi}[naive definition]\label{ope:trans:elliptic:naïve}
We will say that a $G$-equivariant selfadjoint family $A$ of pseudodifferential operators acting on a vector bundle $E=E^+\oplus E^-$ of order $0$, odd for the graduation of $E$ is a family of $G$-transversally elliptic operator or that $A$ is a $G$-transversally elliptic family if 
$$\sup \limits_{m\in M}\|\sigma_A(m,\xi)^2-\rm{id}_{E_m}\|\rightarrow 0$$
when $(m,\xi) \rightarrow \infty$ in $T^V_GM$.

\end{defi}

\begin{remarque}
By a classical argument, the symbol of a $G$-transversally elliptic family  as in Definition \ref{ope:trans:elliptic:naïve} represents a class in the Kasparov $\k$-theory group 
$\k\K(\mathbb{C},C_0(T_G^VM))$ which is simply given by the Kasparov cycle $(C_0(\pi^*E),\sigma_A)$, where $C_0(\pi^*E)$ is the space of continuous sections of $\pi^*E \rightarrow T^V_GM$ vanishing at infinity. Hence, it defines a class in the topological $G$-equivariant $\k$-theory group
$\K (T_G^VM)$.
\end{remarque}

\begin{defi}[technical definition]\label{ope:trans:elliptic:technique}
We will say that a $G$-equivariant selfadjoint family $A$ of pseudodifferential operators of order $0$,  is a family of $G$-transversally elliptic operators or that $A$ is a $G$-transversally elliptic family if its principal symbol $\sigma_A$ satisfies the following condition: \\
$\forall \varepsilon >0, \exists c>0\text{ such that for any } (m, \xi) \in T^VM, \text{ we have}$
$$\|\sigma_A^2(m,\xi)-\rm{id} \|_{(m,\xi)}\leq c(1+q_m(\xi))(1+\|\xi\|^2)^{-1}+\varepsilon$$
where $\|\cdot\|_{(m,\xi)}$ means the operator norm on $E_m$.
\end{defi}

\noindent
The following proposition is stated in \cite{Kasparov:KKindex} in the case of a single operator:

\begin{prop}\label{prop:inégalité:preuve:thm}
Let $T$ be a selfadjoint family of pseudodifferential operators of order $0$ on $M$, with principal symbol $\sigma_T$. Denote by $F$ a selfadjoint family of pseudodifferential operators with principal symbol given by $\sigma_F(m,\xi)=(1+q_m(\xi))(1+|\xi|^2)^{-1} $. Suppose that $\forall \varepsilon >0$, $\exists c>0$ such that $\forall (m, \xi)\in T^VM$
$$
\|\sigma_T(m,\xi)\|\leq c~\sigma_F(m,\xi)+\varepsilon.
$$
Then $\forall \varepsilon >0$, there exist $c_1$, $c_2>0$ and two selfadjoint families of integral pseudodifferential operators with continuous (vertical) kernel such that
$$
-(c_1F+\varepsilon +R_1)\leq T\leq c_2F+\varepsilon +R_2.
$$
\end{prop}

\begin{proof}\ We adapt the classical proof to our setting, see also Appendix \ref{Appendix}. 
For $\varepsilon'>\varepsilon$, we have:
\begin{equation}\label{preuve:liminf}
\varliminf \limits_{\substack{\xi \to \infty \\ \xi \in T^VM^*}}  \mathrm{Re}\big(c\sigma_F(m,\xi)+\varepsilon'-\sigma_T(m,\xi)\big)>0,
\end{equation} 
since $\left\vert\mathrm{Re}\left(\sigma_T(m,\xi)\right)\right\vert\leq \|\sigma_T (m, \xi)\|_{E_m} \; \rm{id}$. Indeed, we have
$$
0< c~\sigma_F(m,\xi)+\varepsilon' -\mathrm{Re}\left(\sigma_T(m,\xi)\right)=\mathrm{Re}\left(c~\sigma_F(m,\xi)+\varepsilon' -\sigma_T(m,\xi)\right).
$$
By Proposition \ref{prop:limiteinffam}, we can find  families of operators $S_2$ and $R_2$ such that 
$$
S_2^*S_2-(cF+\varepsilon'-T)=R_2\text{ and hence } 0\leq -T+cF+\varepsilon' + R_2.
$$
 Replacing $T$ by $-T$, the inequality  \eqref{preuve:liminf} remain true. Hence, we can also find families of operators $S_1$ and $R_1$ such that $0\leq T+cF+\varepsilon'+R_1$. When $\varepsilon'\to \varepsilon$, we get the result.
\end{proof}

Let now $P_0 : C^{\infty ,0}(M,E^+) \rightarrow C^{\infty ,0}(M,E^-)$ be a $G$-invariant family of $G$-transversally elliptic pseudodifferential operators of order $0$. We will denote by $P : C^{\infty ,0}(M,E) \rightarrow C^{\infty ,0}(M,E)$, the pseudodifferential family $\begin{pmatrix}
0&P_0^*\\P_0&0
\end{pmatrix}$.

\begin{remarque}\label{lem:Ppi=piP}
Since $P$ is $G$-invariant, for any  $\varphi \in C^*G$, we have $[\pi(\varphi),P]=0$.
\end{remarque}

We fix, for later use,  a $G$-equivariant continuous family   $\Delta_G$ of second order differential operators (along the fibers) with the principal symbol $q$, see also \cite{atiyah1974elliptic}.
If $X\in \g$ and $s\in C^{\infty ,0}(M,E)$, we denote by  ${\mathscr L} (X)(s)$ the Lie derivative associated with the corresponding action, i.e.
$$
{\mathscr L} (X)(s) := \frac{d}{dt}_{\vert_{t=0}} (e^{tX}\cdot s).
$$
In the same way, using the action of $G$ on itself by left translations, we denote for any given basis $\{V_k\}$ of $\g$ with dual basis $\{v_k\}$ and any $\varphi\in C^\infty (G)$, by $\dfrac{\partial\varphi}{\partial V_k}$  the derivative along the one-parameter subgroup of $G$ corresponding to the vector $V_k$. We then define the family of differential operators 
$$d_G=(d_{G,b })_{b\in B}: C^{\infty ,0}(M,E) \rightarrow C^{\infty ,0}(M,E\otimes \g_M^*),$$
given by $d_G(s)=\sum \limits_k \mathscr{L}(V_k)s\otimes v_k$.\\
Notice that the representation $\pi$ can be extended by setting 
$$
\pi(d\varphi)s=\sum\limits_{k} \int_G \dfrac{\partial\varphi}{\partial V_k}(h)h\cdot s \otimes v_k dh.
$$

\begin{prop}$\cite{Kasparov:KKindex}$\label{prop:dG}
For $\varphi \in C^{\infty}(G)$ and $s\in C^{\infty,0}(M,E)$, we have $d_G(\pi(\varphi ) s)=\pi(d\varphi )s$.
\end{prop}

\begin{proof}
 We have:
\begin{eqnarray*}
d_G(\pi(\varphi)s)&=&\sum \limits_k \mathscr{L}(V_k)\pi(\varphi)s\otimes v_k\\
&=&\sum \limits_k \dfrac{d}{dt}_{\vert_{t=0}} e^{tV_k}\pi(\varphi)s \otimes v_k\\
&=&\sum \limits_k \dfrac{d}{dt}_{\vert_{t=0}} \; \int_G \varphi(g) (e^{tV_k}g)\cdot s \otimes v_k\; dg.
\end{eqnarray*}
Substituting $e^{tV_k}g=h$, we get:
\begin{eqnarray*}
d_G(\pi(\varphi)s)&=&\sum \limits_k \dfrac{d}{dt}_{\vert_{t=0}}\; \int_G \varphi(e^{-tV_k}h) h\cdot s \otimes v_k\; dh\\
&=&\sum \limits_k \int_G \dfrac{\partial\varphi}{\partial V_k}(h) h\cdot s \otimes v_k\; dh\\
&=&\pi(d\varphi )s.
\end{eqnarray*}
\end{proof}
\noindent
We are now in a position to state the main result of this section:

\begin{thm}$\cite{Kasparov:KKindex}$
The triple $\big( \mathcal{E},\pi,P\big)$ is an even  $G$-equivariant Kasparov cycle for the $C^*$-algebras $C^*G$ and $C(B)$. It thus defines a class in $\k\K(C^*G,C(B))$.
\end{thm}

\begin{proof}
The $\Z_2$-grading is clear. By Proposition \ref{prop:pi*rep}, the representation $\pi$ is $G$-equivariant and by Proposition \ref{prop:Pborné}, the operator $P$ is $C(B)$-linear and belongs to  $\mathcal{L}_{C(B)}(\mathcal{E})$.
By Remark \ref{lem:Ppi=piP}, we know that $[\pi(\varphi), P]=0$. Moreover, $P$ is selfadjoint, and it thus remain to check that $(\rm{id}-P^2)\circ \pi(\varphi )\in \mathcal{K}_{C(B)}(\mathcal{E})$. Denote by $\sigma_P$ the  principal symbol of $P$. As $P$ is a family of $G$-transversally elliptic operators, the principal symbol of $\rm{id}-P^2$ which coincides with $\rm{id}-\sigma_P^2$, satisfies the assumption of Proposition \ref{prop:inégalité:preuve:thm}. Therefore $\forall \varepsilon>0$, there exist $c_1$, $c_2>0$ and 
 selfadjoint integrals operators with continuous kernels $R_1$ and $R_2$ such that
$$
-(c_1F+\varepsilon +R_1)\leq 1-P^2 \leq c_2F+\varepsilon +R_2,
$$ 
where $F$ is a continuous family of operators with symbol $\sigma_F(m,\xi)=(1+q_m(\xi))(1+|\xi|^2)^{-1}$. 
We may take for $\Delta_G$ the operator $d^*_Gd_G$. Indeed, the symbol of $d_G$ at $(m,\xi) \in T^VM$ is given by $\sum \limits_{k}i\langle \xi,f'_m(V_k)\rangle \ext (v_k)=\sum \limits_{k}i\langle f'^{*}_{m}(\xi ),(V_k)\rangle \ext (v_k)=i~ \ext(f'^{*}_{m}(\xi))$. So the symbol of $d_G^*$ is given by $-i~\interieur (f'^{*}_{m}(\xi))$. Then we get that the symbol of $d_G^*d_G$ is given by $\langle f'^{*}_{m}(\xi),f'^{*}_{m}(\xi)\rangle =q_m(\xi)$ modulo symbols of order $1$. 
Denote also by $\Delta$ a vertical Laplace operator associated with the fixed metric, that is a family of second order differential operators with principal  symbol $\|\xi\|^2$ for $(m,\xi)\in T^VM^*$.  Modulo families of operators of negative order, the family of operators $F$ coincides with the operator $d_G^*(1+\Delta)^{-1}d_G$.

By Proposition \ref{prop:dG}, we know that $d_G\circ \pi (\varphi ) =\pi(d\varphi )$ is a bounded operator on $\mathcal{E}$. Moreover, $d_G^*(1+\Delta)^{-1}$ has negative order, so by Corollary \ref{cor:compact}, it is a compact operator of the Hilbert module $\mathcal E$. It follow that $d_G^*(\rm{id}+\Delta )^{-1}d_G\pi(\varphi )$ is compact as well. 
In order to show that the operator $(\rm{id}-P^2)\circ \pi(\varphi )$ is compact, we first notice that for $\psi =\varphi^*\star \varphi$, we have: 
$$-\pi(\varphi )^*\circ (c_1F+\varepsilon +R_1)\circ \pi(\varphi )\leq \pi(\varphi )^*\circ (\rm{id}-P^2)\circ \pi(\varphi ) \leq \pi(\varphi )^*\circ (c_2F+\varepsilon +R_2)\circ \pi(\varphi ){,}$$
i.e.
$$- (c_1F+\varepsilon +R_1)\circ \pi(\psi )\leq  (\rm{id}-P^2)\circ \pi(\psi ) \leq (c_2F+\varepsilon +R_2)\circ \pi(\psi ),$$
since all the operators are $G$-invariant. Therefore, passing to the Calking algebra and letting $\epsilon$ tend to $0$, we deduce that $(\rm{id}-P^2)\circ \pi(\psi )$ is compact for any non-negative $\psi\in C^*G$. Now, using continuous functional calculus, we may write any $\varphi \in C^*G$ as a linear combination of non-negative elements and conclude. 

\end{proof}
\noindent
In the following definition and as before, $P=\begin{pmatrix}
0 & P_0^*\\ P_0 & 0
\end{pmatrix} $.

\begin{defi}
The index class $\indm (P_0)$ of a $G$-invariant family $P_0$ of $G$-transversally elliptic operators is defined as:
$$
\indm (P_0) \; :=\; [\mathcal{E},\pi,P]\;\; \in \k\K(C^*G,C(B)).
$$
We also denote by  $\ind(P_0)$ its image in $\k\k(C^*G,C(B))$. 
\end{defi}

\begin{remarque}\
If the family $P_0$ is invariant with respect to an extra action of a compact group $H$, whose action commutes with the action of $G$, then the previous construction yields, for any $G\times H$-equivariant family  a $G$-transversally elliptic operators to an $H$-equivariant index class
$$
\mathrm{Ind}^{\mathrm{M|B}}_{\mathrm{G\times H}}(P_0) \in \k\k_{\mathrm{G\times H}}(C^*G,C(B)).
$$ 
We then  denote accordingly by $\mathrm{Ind}^{\mathrm{M|B}}_\mathrm{H}(P_0)$ its image in $\k\k_{\mathrm{ H}}(C^*G,C(B))$.
\end{remarque}

\subsection{Unbounded version of the index class}
Since the geometric operators play an important part in the index theory of $G$-transversally elliptic operators, we now introduce the definition of the index class for operators of  order $1$, using the unbounded version of Kasparov's theory, see for instance  \cite{baaj1983theorie}. The unbounded version simplifies the computation of some Kasparov products in the next sections.

\begin{defi}[Unbounded Kasparov module $\cite{baaj1983theorie}$]
Let $A$ and $B$ be $C^*$-algebras. An $(A,B)$-unbounded Kasparov cycle $(E,\phi,D)$ is a triple where $E$ is a Hilbert $B$-module, $\phi : A \rightarrow \mathcal{L}(E)$ is a graded $\star$-homomorphism and $(D, \rm{dom} (D))$ is an unbounded regular seladjoint operator such that:
\begin{enumerate}
\item $(1+D^2)^{-1}\phi(a)\in \mathcal{K}(E)$, $\forall a \in A$,
\item  The subspace of $A$ composed of the elements $a\in A$ such that $\phi (a) (\rm{dom}(D))\subset \rm{dom}(D)$ and $[D,\phi(a)]=D\phi (a) - \phi(a) D$ is densely defined and extends to an adjointable operator on $E$,  is dense in $A$.
\end{enumerate}
When $E$ is $\Z_2$-graded with $D$ odd and $\pi (a)$ even for any $a$, we say that the Kasparov cycle is even. Otherwise, it is odd. 
\end{defi}

In \cite{baaj1983theorie}, appropriate equivalence relations are introduced on such (even/odd) unbounded Kasparov cycles, which allowed to recover the groups $\k\k^* (A, B)$. When the compact group $G$ acts on all the above  data, one recovers similarly $\k\k_\mathrm{G}^*(A, B)$ by using the equivariant version of the  Baaj-Julg unbounded cycles of the previous definition.  
We  now  recall some useful properties of unbounded operators on Hilbert modules over our commutative $C^*$-algebra, our main reference is \cite{hilsum1989fonctorialite}. Let $X$ be a locally compact space and let $\mathcal{F}$ be a Hilbert $C_0(X)$-module. We denote by $\mathcal{F}_x$ the fiber of $\mathcal{F}$  at $x\in X$ and denote for any $\xi \in \mathcal{F}$ by $\xi_x\in \mathcal{F}$ its value at $x$. If $\mathcal{G} \subset \mathcal{F}$ is a subset of ${\mathcal F}$, we denote by $\mathcal{G}_x=\{\xi_x, \xi\in \mathcal{G}\}$. The following lemma is proved in \cite{hilsum1989fonctorialite}. 

\begin{lem}$\cite{hilsum1989fonctorialite}$
Let $\mathcal{G}$ be a $C_0(X)$-submodule of $\mathcal{F}$, then 
\begin{enumerate}
\item For any $x\in X$, $\overline{\mathcal{G}_x}=(\overline{\mathcal{G}})_x$. 
\item $\mathcal{G}$ is dense in $\mathcal{F}$ if and only if $\mathcal{F}_x$ is dense in $\mathcal{G}_x$, $\forall x\in X$.
\end{enumerate}
\end{lem}

\noindent
Let, for any  $x\in X$, $T_x$ be a given unbounded operator on $\mathcal{F}_x$ with domain ${\rm{dom}}(T_x)\subset \mathcal{F}_x$. We then define an unbounded operator $T$ on $\mathcal{F}$ by setting
$$
\mathrm{dom} (T)=\{\xi\in \mathcal{F} /\xi_x\in \mathrm{dom} (T_x), \forall x\in X ~\mathrm{and }~ T\xi := (T_x\xi_x)_{x\in X} \in \mathcal{F}\}.
$$
We shall also use the following proposition. 

\begin{prop}\label{regular}$\cite{hilsum1989fonctorialite}$
Assume that for any $x\in X$, the operator $T_x$ is selfadjoint, and that $T$ is densely defined. Then $T$ is selfadjoint and the following  assertions are equivalent: 
\begin{enumerate}
\item $T$ is regular;
\item for any $x\in X$, $(\mathrm{dom}(T))_x$ is an essential domain for $T_x$;
\item for all $x\in X$, we have $(\mathrm{dom}(T))_x=\mathrm{dom}(T_x)$.
\end{enumerate}
\end{prop}

Let us get back now to our situation with the Hilbert $C(B)$-module ${\mathcal E}$ associated with the bundle $E\to M$ and with the previous data.

\begin{defi}
Let $P_0 : C^{\infty,0}(M,E^+)\rightarrow C^{\infty,0}(M,E^-)$ be a family of $G$-invariant pseudodifferential operators of order $1$. We say that $P_0$ is a family of $G$-transversally elliptic operators if its principal symbol $\sigma (P_0) : \pi^*E^+ \rightarrow \pi^*E^-$ is invertible over $T^V_GM\setminus M$.
\end{defi}

Let $P_0:C^{\infty,0}(M,E^+)\rightarrow C^{\infty,0}(M,E^-)$ be a $G$-invariant family of $G$-transversally elliptic pseudodifferential operators of order $1$. Denote by $\hat{P}$ the operator $\begin{pmatrix}
0&P_0^*\\P_0&0
\end{pmatrix}$, where $P_0^*$ is the family $\{(P_0^*)_b\}_b$ corresponding to  the formal adjoints along the fibers.

\begin{prop}\label{Unbounded} 
Let $P_0:C^{\infty,0}(M,E^+)\rightarrow C^{\infty,0}(M,E^-)$ be a $G$-invariant family of $G$-transversally elliptic pseudodifferential operators of order $1$, and let $\hat{P}=\begin{pmatrix}
0&P_0^*\\P_0&0
\end{pmatrix}$ as above. Then with the domain $\rm{dom} (\hat{P}):= \{ \xi \in \mathcal{E},\ \xi_t \in \mathrm{dom}(\hat{P}_t)~\mathrm{and}~\hat{P}\xi \in \mathcal{E}\}$, the unbounded operator $\hat{P} : \mathcal{E}\rightarrow \mathcal{E}$, 
 is essentially seladjoint. So if we denote by $P$ its closure then $P$ is regular and selfadjoint. 
\end{prop}

\begin{proof}
Clearly $\hat{P}$ is then densely defined in $\mathcal{E}$ because $\mathrm{dom}(\hat{P})$ contains $C^{\infty ,0}(M,E)$. For  any $b\in B$, the operator $\hat{P}_b$ with domain $C^{\infty}(M_b,E_b)$ is essentially selfadjoint, see \cite{fox1994index}. As $\mathrm{dom}(P) = \{\xi \in \mathcal{E},\ \xi_b \in \mathrm{dom}(\overline{P_b})\ \mathrm{et}\  \overline{P}\xi \in \mathcal{E}\}$, we get by Proposition \ref{regular} that $P$ is selfadjoint and it remains to check that $P$ is regular which again by Proposition \ref{regular} reduces to proving that for any $b\in B$,  $(\mathrm{dom}(P))_b$ is an essential domain for $P_b$. We have $C^{\infty}(M_b,E_b) \subset (\mathrm{dom}(P))_b$, so $P$ is regular.
\end{proof}

We are now in position to state our result. Recall that $G$ acts on $C^*G$  by conjugation, that $\mathcal{E}$ is the $G$-equivariant Hilbert $C(B)$-module associated with the continuous field of Hilbert spaces $\{L^2 (M_b, E_b)_{b\in B}\}$, and that $\pi : C^*G \rightarrow \mathcal{L}_{C(B)}(\mathcal{E})$ is the representation defined in Proposition \ref{prop:pi*rep}. 

\begin{thm}
Let $P_0:C^{\infty,0}(M,E^+)\rightarrow C^{\infty,0}(M,E^-)$ be a $G$-invariant family of $G$-transversally elliptic pseudodifferential operators of order $1$, and let $P$ be the associated regular self-adjoint operator defined above out of $\begin{pmatrix}
0&P_0^*\\P_0&0
\end{pmatrix}$. Then the triple $(\mathcal{E}, \pi ,P)$ is an even $(C^*G,C(B))$-unbounded Kasparov cycle, which defines a class in $\k\K(C^*G,C(B))$.
\end{thm}

\begin{proof}
By Proposition \ref{prop:pi*rep}, the representation $\pi$ is $G$-equivariant. For any $\varphi\in C^\infty (G)$, it is easy to see that $\pi(\varphi)$ preserves the domain of $P$ and by Remark \ref{lem:Ppi=piP}, we have $[\pi(\varphi),P]=0$. 
%De plus, $P$ est auto-adjoint donc $(P^*-P)\circ \pi(\varphi)=0$.
It remains to check that $(1+P^2)^{-1}\circ\pi(\varphi) \in \mathcal{K}(\mathcal{E})$. We may take for our operator $\Delta_G$ for instance the Casimir operator, which is a differential operator along the fibers of order $2$. The operator $1+P^2+\Delta_G$ is elliptic along the fibers since  $\sigma (P_b)$ is invertible in the transversal direction to the orbits and $\sigma (\Delta_G)$ is invertible in the orbits direction, see \cite{atiyah1974elliptic}. We hence deduce that the resolvent $(1+P^2+\Delta_G)^{-1}$ is a  compact operator in $\mathcal{E}$. 

We now show that for any $\varphi\in C^\infty (G)$, the operator 
$$
(1+P^2+\Delta_G)^{-1}\circ \pi(\varphi)-(1+P^2)^{-1}\circ \pi (\varphi ), 
$$ 
is compact, which will insure that $(1+P^2)\circ \pi(\varphi )$ is also compact. 
Denote again by $\Delta$ the Laplacian on $G$ viewed as a riemannian $G$-manifold. Using that $\Delta_G\pi(\varphi)=\pi(\Delta\phi)$ and that $[\pi(\varphi) ,P]=0$, we have: 
\begin{multline*}
(1+P^2+\Delta_G)^{-1}\circ \pi(\varphi)-(1+P^2)^{-1}\circ \pi (\varphi )=-(1+P^2+\Delta_G)^{-1}\Delta_G(1+P^2)^{-1} \circ\pi(\varphi )\\
~~~~~~=-(1+P^2+\Delta_G)^{-1}\pi(\Delta\varphi)(1+P^2)^{-1}.
\end{multline*}
Now since $\pi(\Delta\varphi)$ and $(1+P^2)^{-1}$ are adjointable operators and since $(1+P^2+\Delta_G)^{-1}$ is compact, we deduce that $(1+P^2+\Delta_G)^{-1}\circ \pi(\varphi)-(1+P^2)^{-1}\circ \pi (\varphi )$ is a compact operator on ${\mathcal E}$.
\end{proof}

\begin{defi}\label{def:indice:NB}
The index class $\indm(P_0)$ of a $G$-invariant family $P_0$ of $G$-transversally elliptic pseudodifferential operators of positive order,  is the class in $\k\K(C^*G,C(B))$ of the $(C^*G,C(B))$-unbounded Kasparov cycle $(\mathcal{E},\pi, P)$, i.e.
$$
\indm(P_0)=[\mathcal{E},\pi, P] \quad \in \k\K(C^*G,C(B)).
$$
\end{defi}

\begin{remarque}
The relation with the bounded version is made using the Woronowicz transformation, see \cite{baaj1983theorie}. More precisely, if $P_0$ is a $G$-invariant family of $G$-transversally elliptic pseudodifferential operators of  order $1$, then the following triple $(\mathcal{E},\pi, P(1+P^2)^{-1/2})$ is a (bounded) Kasparov cycle and the index class coincides with  $[\mathcal{E},\pi, P(1+P^2)^{-1/2}]$.
\end{remarque}

\subsection{$\k$-theory Multiplicity}
\noindent
In the present paragraph, we shall only need the bounded version of the index class. Recall that we also have the index class $\ind(P)$ in $\k\k(C^*G,C(B))$ obtained by forgetting the $G$-action on $C^*G$.  For any irreducible unitary representation $V$ of $G$, we now define a class $m_P(V) \in \k (B) $ corresponding to the multiplicity of $V$ in the index class $\ind(P)$, and prove its compatibility with a second definition using the obvious Kasparov's intersection product.

\noindent
So let us fix an  irreducible unitary representation  $V$ of $G$.
Let $\mathcal{E}^G_V$ be the Hilbert $C(B)$-submodule of $V\otimes \mathcal{E}$ composed of the $G$-invariant elements, i.e.  $\mathcal{E}^G_V:=(V\otimes \mathcal{E})^G$.  In term of continuous fields of Hilbert spaces, the Hilbert module $\mathcal{E}^G_V$ corresponds to the field $\big((L^2(M_b,V\otimes E_b)^G)_{b\in B}$ with the space of continuous sections being the space  of  $G$-invariant continuous sections of the field associated with $V\otimes E$. If $C_{v, w}$ is the coefficient of the representation $V$ which corresponds to the vectors $v, w\in V$, then
the inner product satisfies the formula
$$
\langle v\otimes s,w\otimes t\rangle = \langle s,\pi(C_{v,w})t\rangle.
$$
Indeed, we have
\begin{eqnarray*}
\langle v\otimes s,w\otimes t\rangle (b)\hspace*{-0.3cm}&=&\hspace*{-0.4cm}\int_{M_b} \langle \int_G (g\cdot (v\otimes s))(m) \; dg ,\int_G (h\cdot (w\otimes t))(m)dh\rangle_ {(\underline{V}\otimes E)_m}\; d\mu_b(m)\\
&=&\int_{M_b} \int_G\int_G\langle  g\cdot v,h\cdot w\rangle_ {V} \langle (g\cdot s) (m), (h\cdot  t)(m)\rangle_ { E_m}\,dg\,dh\; d\mu_b(m)\\
&=&\int_{M_b}\langle s(m),\pi(C_{v,w})t(m)\rangle_ {E_m}\; d\mu_b(m)\\
&=&\langle s,\pi(C_{v,w})t\rangle (b).
\end{eqnarray*}
By restricting the operator $\text{id}_V \otimes P$ to the $G$-invariant elements, we get the adjointable  operator $P^G_V \in \mathcal{L}_{C(B)}(\mathcal{E}_V^G)$, i.e.
$$
P^G_V:=(\text{id}_V \otimes P)_{|\mathcal{E}^G_V}.
$$ 

\begin{lem}
$\big( \mathcal{E}^G_V,P^G_V\big)$ is an even Kasparov cycle which defines a class in $\k\k(\mathbb{C},C(B))$.
\end{lem}

\begin{proof}
We only need  to show that $(P^G_V)^2-1$ is compact. We have $(P_V^G)^2-1=(\text{id}_V\otimes (P^2-1))_{|\mathcal{E}_V^G} $. Let $(e_1, \cdots ,e_n)$ be a orthonormal basis of $V$. Let $s=\sum \limits_i e_i\otimes s_i$ be an element of $\mathcal{E}_V^G$. Then we have $\int_G g\cdot s=s$ so that an easy computation gives
$$
s=\sum \limits_{i,j}e_j\otimes \pi(C_{e_j,e_i}) s_i.
$$
But 
\begin{eqnarray*}
((P^G_V)^2-1)s&= &(\text{id}_V\otimes (P^2-1))(\sum \limits_{i,j}e_j\otimes \pi(C_{e_j,e_i})s_i)\\
&=&\sum \limits_{i,j} e_j\otimes (P^2-1)\circ \pi(C_{e_j,e_i})s_i\\
&=&\sum \limits_{t,j}e_j\otimes e_t^*\otimes (P^2-1)\circ \pi(C_{e_j,e_t})(s).
\end{eqnarray*}
Thus it follows that $(P^G_V)^2-1 = \sum \limits_{t,j}e_j\otimes e_t^*\otimes (P^2-1)\circ \pi(C_{e_j,e_t})$.
Moreover, we know that $(\dim V) C_{e_j,e_j}\star C_{e_j,e_t}=C_{e_j,e_t}$ and that $(P^2-1)\circ \pi(C_{e_j,e_j})$ are compact operators. Let $(R_n^j)_n$ be a sequence of finite rank operators on $\mathcal{E}$ which converges to $(P^2-1)\circ \pi(C_{e_j,e_j})$. The sequence of finite rank operators $(e_j\otimes R_n^j\circ \pi(C_{e_j,-}))$ then converges to $e_j\otimes (P^2-1)\pi(C_{e_j,e_j})\circ \pi(C_{e_j,-})$ for the operator norm induced by the inner product on $\mathcal{E}_V^G$. Indeed, we have: 
\begin{multline*}
\left\langle e_j\otimes \big((P^2-1)\circ \pi(C_{e_j,e_j})-R_n^j\big)\circ \pi(C_{e_j,-}) (s), \; s' \right\rangle_ {(V\otimes \mathcal{E})^G}\\
=\left\langle  \big((P^2-1)\circ \pi(C_{e_j,e_j})-R_n^j\big)\sum \limits_k \pi(C_{e_j,e_k}) s_k, \sum \limits_t \pi(C_{e_j,e_t}) s'_t\right\rangle_ {\mathcal{E}},
\end{multline*}
\end{proof}

\begin{defi}
The $\k$-multiplicity $m_P(V)$ of the irreducible unitary representation $V$ of $G$ in the index class  $\ind (P_0)$ is  the image of the class 
$[(\mathcal{E}_V^G,P_V^G)]\in \k\k(\mathbb{C},C(B))$ under the isomorphism $\k\k(\mathbb{C},C(B))\cong \k(B)$. So $m_P(V)$ is the class of a virtual vector bundle over $B$, an element of the topological $\k$-theory group   $\k(B)$. 
\end{defi}

Recall that the representation $V$ defines a class in the $\k$-theory of the $C^*$-algebra $C^*G$, and hence an element, denoted by  $[V]$, of the Kasparov group $\k\k(\mathbb{C}, C^*G)$. See \cite{julg1981produit:croise} for the details.

\begin{prop}
The class $[(\mathcal{E}_V^G,P_V^G)]$ coincides (as expected) with the Kasparov product
$$
[V]\underset{C^*G}{ \otimes}\ind(P_0)\in \k\k(\mathbb{C},C(B)).
$$
Thus the $\k$-multiplicity of $V$ in the index of $P_0$ is the element of $\k(B)$ corresponding to the Kasparov product $[V]\underset{C^*G}{ \otimes}\ind(P_0)$, through the isomorphism $\k\k(\mathbb{C},C(B))\cong \k(B)$.
\end{prop}

\begin{proof}
The class $[V]$ is given by the cycle $[(V,0)] \in \k\k(\mathbb{C},C^*G)$.  The operator $\text{id}_V\underset{C^*G}{ \otimes} P$ makes sense because $P$ commute with the action of $C^*G$. Using the $C^*G$-module structure on $V$, it follows that $\forall v\in V$, $v=\sum\limits_k e_k \cdot C_{e_k,v}$. Therefore $\forall s=\sum\limits_i v_i\otimes s_i \in V\underset{C^*G}{ \otimes}\mathcal{E}$, we have $s=\sum \limits_k e_k\otimes \sum \limits_i \pi(C_{e_k,v_i})s_i$. Using again the relation $(\dim V) C_{e_j,e_j}\star C_{e_j,v}=C_{e_j,v}$, we deduce again  that $(\text{id}_V\underset{C^*G}{\otimes} P)^2-1$ is compact. Moreover, the $C(B)$-modules $V\underset{C^*G}{\otimes} \mathcal{E}$ and $(V\otimes \mathcal{E})^G$ are isomorphic. The isomorphism is given by 
\begin{eqnarray*}
V\underset{C^*G}{\otimes}\mathcal{E} &\longrightarrow &(V\otimes \mathcal{E})^G\\
\sum \limits v_i\otimes s_i& \longmapsto &\sum \limits_i\int_{G}g\cdot (v_i\otimes s_i)\; dg.
\end{eqnarray*}
This map is well defined because we have 
$$
\int_G g\cdot (v\cdot \varphi \otimes s) dg = \int_G g\cdot(v\otimes \pi(\varphi)s) \; dg.
$$
But we have already seen that $\int_G g\cdot(v\otimes s)dg=\sum \limits_k e_k  \otimes \pi(C_{e_k,v})s$.
Therefore we have proved that $[\mathcal{E}_V^G,P_V^G]=[V]\underset{C^*G}{ \otimes}\ind(P_0)$.
\end{proof}

\section{The properties of the index class}
\noindent As before, we denote by $T^V_GM :=T^VM\cap T_GM$. 

\subsection{The index map}

\begin{prop}
The index class $\indm(P)$ only depends on the class $[\sigma(P)]\in \K(T^V_GM)$ of the principal symbol $\sigma (P)$, and the index map induces a group homomorphism:
$$\indm : \K(T^V_GM) \rightarrow \k\k_{\mathrm{G}}(C^*G,C(B)).$$
More precisely, the map $[\sigma] \mapsto \indm(P(\sigma))$ is well defined, when we choose some  quantization $P(\sigma)$  of $\sigma$. 
\end{prop}

\begin{proof} This is classical and we follow \cite{Atiyah-Singer:I} and \cite{atiyah1974elliptic}. Let ${\mathcal C}(T^V_GM)$ be the semigroup of homotopy classes of transversally elliptic symbols of order $0$ and ${\mathcal C}_{\phi}(T^V_GM)\subset {\mathcal C}(T^V_GM)$ the subspace composed of classes of symbols whose restriction to the sphere bundle of $T^V_GM$, is induced by a bundle isomorphism over $M$. We know that $\K(T^V_GM):={\mathcal C} (T^V_GM)/{\mathcal C}_{\phi}(T^V_GM)$, \cite{Atiyah-Singer:I}.
It thus suffices  to show that:
\begin{enumerate}
\item if $[\sigma]=[\sigma ']$ in ${\mathcal C}(T^V_GM)$ then $\indm \sigma =\indm \sigma '$;
\item for $P $ and $Q$ families of $G$-transversally elliptic operators, we have $\indm (P\oplus Q)= \indm (P)+\indm(Q)$,
\item if $P$ has symbol which is induced by a bundle isomorphism over $M$ then $\indm(P)=0$.
\end{enumerate}
1. Let $\sigma_t$ be a homotopy between $\sigma $ and $\sigma'$ composed of transversally elliptic symbols along the fibers, the quantization of this homotopy give an operator homotopy, and hence we end up with   the same class in $\k\K(C^*G,C(B))$, by definition of the Kasparov group $\k\K(C^*G,C(B))$.\\
2. If $P : C^{\infty ,0}(M,E) \rightarrow C^{\infty ,0}(M,E)$ and $Q: C^{\infty ,0}(M,F) \rightarrow C^{\infty ,0}(M,F)$ are families of $G$-transversally elliptic operators, then $\indm(P\oplus Q )=[(\mathcal{E}\oplus \mathcal{F}, \pi_{\mathcal{E}}\oplus \pi_{\mathcal{F}} ,P\oplus Q)]=[(\mathcal{E}, \pi_{\mathcal{E}} ,P)]+[(\mathcal{F},\pi_{\mathcal{F}}, Q)]=\indm(P)+\indm(Q)$.\\
3. This is clear, see for instance \cite{Atiyah-Singer:IV}. 
\end{proof}

\subsection{The case of free actions }
Let $G$ and $H$ be compact Lie groups. Let $p: M \rightarrow B$ be a $G\times H $-equivariant fibration of compact manifolds. We suppose that the action of $G\times H$ is trivial on $B$ and that $H$ acts freely on $M$. Denote by $\pi^{H}: M \rightarrow M/H$ the $G$-equivariant projection. By \cite{atiyah1974elliptic}, we know that $\pi^H$ induces an isomorphism
$$\pi^{H*} : \K(T^{V}_G(M/H)) \rightarrow \text{K}_{G\times H}(T^{V}_{G\times H}M)$$ 
which identifies the classes of symbols of $G$-invariant families of $G$-transversally elliptic operators on $M/H \rightarrow B$ with those of symbols of $G\times H$-invariant families of $G\times H$-transversally elliptic operators on $M\rightarrow B$. Denote by $\pi : T^VM \rightarrow M$ and by $\bar{\pi} : T^V(M/H)\rightarrow M/H$ the projections. The inverse map of $\pi_H^*$ is given as follow. Let $\sigma : \pi^*E^+ \rightarrow \pi^*E^-$ be a $G\times H$-transversally elliptic symbol along the fibers on $M$. Then we define a $G$-equivariant symbol $\overline{\sigma} : \bar{\pi}^*(E^+/H) \rightarrow \bar{\pi}^*(E^-/H)$ on $T^V(M/H)$ which is $G$-transversally elliptic along the fibers by setting $\overline{\sigma}(\bar{m},\bar{\xi})(\bar{v})=\sigma(m,\xi)(v)$, where $\bar{m}=\pi^H(m)\in M/H$, $\bar{\xi}=\pi^H(\xi)\in T^V_{\bar{m}}(M/H)$, $\bar{v}=\pi^H(v)\in (E/H)_{\bar{m}}$.

Let $\hat{H}$ be the space of isomorphism classes of unitary irreducible representations of $H$. Denote by $\chi_\alpha$ the character associated with the representation $\alpha$ and $W_\alpha$ the corresponding vector space. We have a natural map:
$$\begin{array}{clc}
R(H\times G) &\rightarrow &\K(M/H)\\
V &\mapsto &\underline{V}^* 
\end{array}$$
where $\underline{V}^*=M \times_H V^*$ is the homogeneous $G$-vector bundle associated to $V^*$ on $M/H$.

\begin{thm}$\cite{atiyah1974elliptic}$\label{thm:action:libre}
The following diagram is commutative:
 $$\xymatrix{\K(T^V_G(M/H))\ar[r]^{\pi^{H*}} \ar[d]_{\underline{W}_{\alpha}^*\otimes }~&~\mathrm{K}_{\mathrm{G\times H}}(T^V_{G\times H}M)\ar[r]^{\hspace*{-0.5cm}\mathrm{Ind}_{}^{\mathrm{ M|B}}}~&~\k\k(C^*(G\times H),C(B))\ar[d]^{\chi_\alpha \otimes_{C^*H}}\\
\K(T^V_G(M/H))\ar[rr]_{\mathrm{Ind}^{\mathrm{M/H|B}}} ~&&~\k\k(C^*G,C(B)). 
 }$$
In other words, if $a\in \K(T^{*V}_G(M/H))$ then $$\mathrm{Ind}^{\mathrm{M/H|B}}( \underline{W}_{\alpha}^*\otimes a)=\chi_{\alpha } \otimes_{C^*H} \ind (\pi^{H*}a),$$ 
\end{thm}

\begin{remarque}
In particular, as an element of $\mathrm{Hom}\big(R(H),\k\k(C^*G,C(B))\big)$ we have
 $$\mathrm{Ind}^{\mathrm{M|B}} (\pi^{H*}a)=\sum \limits_{\alpha \in \hat{H}}\hat{\chi}_\alpha \otimes \mathrm{Ind}^{\mathrm{M/H|B}}( \underline{W}_{\alpha}^*\otimes a) , $$
where  $\hat{\chi}_{\alpha }$ is the element associated to $W_{\alpha}$ in $\k\k(C^*H,\mathbb{C})$.
\end{remarque}

\begin{proof}
Let $\tilde{A}_0 : C^{\infty , 0}(M,\pi^{H*}E^+)\rightarrow C^{\infty , 0}(M,\pi^{H*}E^-)$ be a $G\times H$-invariant family of $G\times H$-transversally elliptic pseudodifferential operators  associated to $\pi^{H*}a\in \Kh(T^{*V}_{G\times H}M)$. Denote by $\pi^{H*}\mathcal{E}$ the $C^*$-completion of $C^{\infty , 0}(M,\pi^{H*}E)$ as a Hilbert $C(B)$-module.
We have:
$$\ind(\pi^{H*}a)=[\pi^{H*}\mathcal{E}, \pi_{G\times H},\tilde{A}]=\sum \limits_{\alpha\in \hat{H}}\hat{\chi}_{\alpha}\otimes [(W_\alpha \otimes \pi^{H*}\mathcal{E})^H,\pi_G, (id_{W_\alpha^*}\otimes \tilde{A})_{|G-inv}],$$
where $\tilde{A}$ is the self-adjoint operator associated with $\begin{pmatrix}
0&\tilde{A}_0^*\\ \tilde{A}_0&0
\end{pmatrix}$ as in Proposition \ref{Unbounded}.  
We shall use the following isomorphism:
$$\begin{array}{clc}
C^{\infty ,0}(M,W_\alpha^* \otimes \pi^{H*}E)^H &\longrightarrow &C^{\infty ,0}(M/H,\underline{W}_{\alpha}^* \otimes E)\\
s&\longmapsto &s_{M/H},
\end{array}$$
given by $s_{M/H}(x)=[m,s(m)]$, for $x=\pi^H(m)$ which allows to identify  $(W_{\alpha}^*\otimes \pi^{H*}\mathcal{E})^H$ with $\overline{C^{\infty ,0}(M/H,\underline{W}_{\alpha}^* \otimes E)}^{C(B)}$. \\
It remains to show that for a good choice of $id_{W_\alpha^*}\otimes \tilde{A}$, the family of operators $id_{\underline{W}_{\alpha}^* } \otimes A$ is a family of pseudodifferential operators associated to $\underline{W}_{\alpha}^* \otimes a \in \K(T^{*V}_G(M/H))$.\\
Let $A_1 : C^{\infty , 0}(M/H,\pi^{H*}E^+)\rightarrow C^{\infty ,0}(M/H,\pi^{H*}E^-)$ be a family of pseudodifferential operators of order $1$ associated to $a$. Let us choose a trivializing system $\{O_j\}$ of $\pi^H : M \rightarrow M/H$ and a subordinate partition of unity $\{\phi_j^2\}$. Let $M_j =(\pi^{H})^{-1}(O_j)\simeq O_j\times H$, we denote by $A^j_1=\phi_jA_1\phi_j$ the induced family of pseudodifferential operators of order $1$ over $O_j$, it lifts to a family of pseudodifferential operators $\tilde{A}_1^j \in \mathcal{P}^1(M_j,\pi^{H*}E)$ on $M_j$ \cite{Atiyah-Singer:IV}. Thanks to the partition $\phi_j$, this family of pseudodifferential operators can be gathered into  a family of pseudodifferential operators on $M$. By averaging the sum of the $\tilde{A}_1^j $ over $G$, we obtain:
$$\tilde{A}=\int_G g\cdot (\sum \limits_j \tilde{A}^j_1)\; dg, $$
which is a family of pseudodifferential operators of order $1$ on $M$. Using that the symbol commutes with lifting and averaging, we see  that $\sigma(\tilde{A})$ represents $\pi^{H*}a$. Moreover,  the restriction of $\tilde{A}$ to the $H$-invariant sections is just the $G$-invariant family $A=\int_G g\cdot (\sum \limits_j A^j_1)\; dg $ of order $1$. \\
Therefore the family $A : C^{\infty , 0}(M/H,E^+)\rightarrow C^{\infty , 0}(M/H,E^+)$ of pseudodifferential operators of order $1$ satisfies $[\sigma(A)]=a$, and lifts to a family of pseudodifferential operators $\tilde{A}$, of order $1$ with $[\sigma(\tilde{A})]=\pi^{H*}a$.  By restricting the map $\tilde{A}$ into the map $\tilde{A}: C^{\infty , 0}(M,\pi^{H*}E^+)\rightarrow C^{\infty ,0}(M,\pi^{H*}E^-)$, we recover $A$. So we have $[(\pi^{H*}\mathcal{E})^H,\pi_G, \tilde{A}_{|H-inv}]=\ind (a)$, by replacing $a$ by $\underline{W}_{\alpha}^*\otimes a$, we get the result.
\end{proof}

\begin{cor}
Let $p : M\rightarrow B$ be a $G\times H$-equivariant fibration. Suppose that the action of $G\times H$ is trivial on $B$ and that the action of $H$ is free on $M$. Then for $a\in \K(T^{*V}(M/H))$, we have the following equality:
$$\chi_\alpha \otimes_{C^*G}\indm (\pi^{H*}a)=\indmsurh(\underline{W}_{\alpha}^*\otimes a ),$$
where $\pi^{H} : M\rightarrow M/H$ is the projection and $ \indmsurh(\underline{W}_{\alpha}^*\otimes a )$ is given by the families index along the fibers of $M/H\rightarrow B$.\\
\end{cor}

\noindent
In particular, as an element of 
$\mathrm{Hom}\big(R(H),\k\k(\mathbb{C},C(B))\big)$ we have:
 $$\indm (\pi^{H*}a)=\sum \limits_{\alpha \in \hat{H}}\hat{\chi}_\alpha \otimes \indmsurh( \underline{W}_{\alpha}^*\otimes a) , $$
where $\hat{\chi}_{\alpha }$ is the element associated to $W_{\alpha}$ in $\k\k(C^*H,\mathbb{C})$.

\subsection{Multiplicative property}
Let $G$ and $H$ be two compact Lie groups.
Let $p:M \rightarrow B$ be a $G$-equivariant fibration of compact manifolds and let $p' : M'\rightarrow B'$ be a $G\times H$-equivariant fibration of compact manifolds. We suppose that $G$ acts trivially on $B$ and that $G$ and $H$ act trivially on $B'$. If $B=B'$ then the fiber product $M\times_B M'$ is a  $G\times H$-equivariant fibration on $B$. 
The purpose of this section is to compute the index of the product $a\sharp b \in \Kh(T^{*V}_{G\times H}(M\times_BM'))$ of $a\in \K(T^{*V}_GM)$ and $b\in \text{K}_{G\times H}(T^{*V}_HM')$ in terms of the index of $a$ and the index of $b$.
Recall that the product of $a$ and $b$ is given by the sharp-product $a\sharp b=\begin{pmatrix}a\otimes 1&-1\otimes b^*\\1\otimes b &a^*\otimes 1 \end{pmatrix}$, that we restrict to $T^{*V}_{G\times H}(M\times_BM')$  \cite{atiyah1974elliptic}, see too \cite{Atiyah-Singer:I}. This product is well defined since any element of $\K(T^{*V}_GM)$ can be represented by a symbol on $T^VM$ which is invertible on $ T^{V}_GM\setminus M $, see again \cite{atiyah1974elliptic}.\\
\medskip
Let $\Delta_B : B \rightarrow B\times B$ be the diagonal map. We denote by $[\Delta_B]$ the Kasparov element in $\k\k(C(B)\otimes C(B),C(B))$ associated to $(C(B), \Delta_B^*,0)$.\\
Recall the Kasparov descent \cite{Kasparov1988}. Let $A$ and $D$ be $G$-$C^*$-algebras. For a given $G$-$C^*$-algebra $A$ we recall the crossed product $C^*$-algebra $A\rtimes G$. More precisely, the space $C(G,A)$ is an involutive algebra for the rules
$$a_1\cdot a_2(t)=\int_G a_1(s)\cdot s(a_2(s^{-1}t)) ds \mathrm{~and~} a^*(t)=t(a(t^{-1}))^*,~a,~a_1,~a_2 \in C(G,A).$$
The completion of the image of $C(G,A)$ by the left regular representation, $C(G,A) \rightarrow \mathcal{L}(L^2(G,A))$ given by 
$$(a\cdot l)(t)=\int_G t^{-1}(a(s))\cdot l(s^{-1}t)ds$$
where $a\in C(G,A)$, $l\in L^2(G,A)$, is the $C^*$-algebra $A\rtimes G$. Let $\mathcal{E}$ be a $G$-$D$-Hilbert module. Define a right action and a scalar product by 
$$e\cdot d(s)=\int_G e(t)t(d(t^{-1}s))dt \mathrm{~and~} \langle e_1,e_2\rangle  (s)=\int_G t^{-1}(\langle e_1(t),e_2(ts)\rangle  _{\mathcal{E}})dt,$$
where $e$, $e_1$, $e_2 \in C(G,\mathcal{E})$ and $d\in C(G,D)$ then the completion of $C(G,\mathcal{E})$ with respect to this Hilbert structure defines a $D\rtimes G$-Hilbert module. If $\pi : A \rightarrow \mathcal{L}_D(\mathcal{E})$ is a $\ast$-morphism then the map $\pi^{\rtimes G } :A\rtimes G \rightarrow \mathcal{L}_D(\mathcal{E}\rtimes G)$ given by $(\pi^{\rtimes G}(a)e)(t)=\int_G\pi(a(s)) s(e(s^{-1}t))ds$ is a $\ast$-morphism. If $T\in \mathcal{L}_D(\mathcal{E})$ then $T$ induces an operator $\tilde{T}\in \mathcal{L}_{D\rtimes G}(\mathcal{E}\rtimes G)$ defined by $(\tilde{T}e)(s)=T(e(s))$.\\

\begin{defi}
The Kasparov descent map $j^G : \k\K(A, D) \rightarrow \k\k(A\rtimes G,D\rtimes G)$ is given by $[(\mathcal{E},\pi ,T)]\rightarrow [(\mathcal{E}\rtimes G,\pi^{\rtimes G},\tilde{T}]$.
\end{defi}

\begin{thm}\label{thm:multiplicativité:indice}
Let $a\in \K(T^{*V}_GM)$ and $b\in \Kh(T^{*V}_HM')$. When $B=B'$ we denote by $i : M\times_B M' \hookrightarrow M\times M'$ the inclusion. 
\begin{enumerate}
\item[1.] The index of $a\sharp b \in \Kh(T^{*V}_{G\times H}(M\times M'))$ is given by the following product:
$$\mathrm{Ind}^{\mathrm{M\times M'|B\times B'}} (a\sharp b)=j^G\indm(b)\otimes_{C^*G} \mathrm{Ind}^{\mathrm{M'|B'}}(a).$$
\end{enumerate}
The inclusion $T^{*V}_{G\times H}(M\times_BM') \hookrightarrow T^{*V}_{G\times H}(M\times M')$ is a $\mathrm{(}G\times H$-equivariant$\mathrm{)}$ proper inclusion.
\begin{enumerate}
\item[2.] The index of $i^*(a\sharp b) \in \Kh(T^{*V}_{G\times H}(M\times_BM'))$ is given by:

$$\mathrm{Ind}^{\mathrm{M\times_BM'|B}} (i^*a\sharp b)=\mathrm{Ind}^{\mathrm{M\times M'|B}} (a\sharp b)\otimes_{C(B\times B)}[\Delta_B].$$

\end{enumerate}

\end{thm}

\begin{proof}~\\
\hspace*{0.8cm}1. Recall that if $P_0$ is a family of pseudodifferential operators of positive order then we denote by $\hat{P}$ the family $\begin{pmatrix}
0&P_0^*\\P_0&0
\end{pmatrix}$ and by $P$ the closure of $\hat{P}$.\\
Let $A_0 : C^{\infty,0}(M,E^+_2) \rightarrow C^{\infty,0}(M,E^-_2)$ be a $G$-invariant family of $G$-transversally elliptic operators of order $1$ associated to  $a $ and $(\mathcal{E}_2,\pi_G ,A)$ the $(C^*G,~C(B'))$ unbounded Kasparov cycle associated (see Definition \ref{def:indice:NB}). Let $B_0 : C^{\infty,0}(M,E^+_1) \rightarrow C^{\infty,0}(M,E^-_1)$ be a $G\times H$-invariant family of $H$-transversally elliptic operators of order $1$ associated to $b$ and $(\mathcal{E}_1,\pi_H ,B)$ the associated $(C^*H,~C(B))$ unbounded Kasparov cycle.\\ 
Denote by $Q(T)=T(1+T^2)^{-1/2}$ the Woronowicz transformation of a selfadjoint regular operator $T$. 
Recall that the index classes associated respectively to $(\mathcal{E}_1,\pi_G ,A)$ and $(\mathcal{E}_2,\pi_H ,B)$ are respectively $\mathrm{Ind}^{\mathrm{M'|B}'}(A)=[(\mathcal{E}_2, \pi_G,Q(A))]$ and $\indm(b)=[(\mathcal{E}_2, \pi_H,Q(B))]$. Notice that the operator $1\underset{\pi_G}{\otimes}A$ is well defined since $A$ commutes with $\pi_G$. We will show that the cycle $(\mathcal{E}_1\rtimes G \underset{\pi_G}{\otimes}\mathcal{E}_2 , \pi_H^{\rtimes G} \underset{\pi_G}{\otimes} 1 ,Q(\tilde{B}\underset{\pi_G}{\otimes}1+1\underset{\pi_G}{\otimes}A))$ is the Kasparov product of $j^G\indm(B)$ and $ \mathrm{Ind}^{\mathrm{M'|B'}}(A)$.\\
Notice that $B\otimes 1 + 1\otimes A$ is the operator used to define the index class as in Definition \ref{def:indice:NB} because $\hat{B}\otimes 1 +1 \otimes \hat{A}$ is essentially selfadjoint.
We have to check that $\mathrm{Ind}^{\mathrm{M\times M'|B\times B'}}_{G} (A\sharp B)=[(\mathcal{E}_1\otimes \mathcal{E}_2, \pi_{G\times H} , B \otimes 1 + 1\otimes A)]$ is given by $[(\mathcal{E}_1\rtimes G \underset{\pi_G}{\otimes}\mathcal{E}_2 , \pi_H\rtimes G \underset{\pi_G}{\otimes} 1 ,Q(\tilde{B}\underset{\pi_G}{\otimes}1+1\underset{\pi_G}{\otimes}A))]$.

Let $U : C^{\infty}(G,C^{\infty,0}(M ,E_1))\otimes C^{\infty ,0}(M',E_2) \rightarrow C^{\infty, 0}(M\times M' ,E_1\boxtimes E_2 )$ be the map defined by $U(s\otimes s')=\int_G s(g)\otimes g\cdot s' dg$. 
We easily check that $U(s \cdot \varphi \otimes s' )=U(s\otimes \pi_G(\varphi )s')$, $\forall s\in C^{\infty}(G, C^{\infty , 0}(M,E_1))$, $\varphi \in C^{\infty}(G)$ and $s'\in C^{\infty ,0}(M',E_2)$, so $U$ induces a map still denoted by $U$ from $C^{\infty}(G,C^{\infty,0}(M ,E_1))\underset{\pi_G}{\otimes} C^{\infty ,0}(M',E_2)$ to $ C^{\infty, 0}(M\times M' ,E_1\boxtimes E_2 ) $. 
This map extends to a unitary isomorphism from $\mathcal{E}_1\rtimes G\underset{\pi_G}{\otimes}\mathcal{E}_2$ to $\mathcal{E}_1\otimes \mathcal{E}_2$. Indeed, $U(C^{\infty}(G,C^{\infty,0}(M ,E_1))\underset{\pi_G}{\otimes} C^{\infty ,0}(M',E_2))$ is dense in $\mathcal{E}_1\otimes \mathcal{E}_2$ and we have the equality 
$\displaystyle \langle U(s_1\underset{\pi_G}{\otimes}s_2),U(s^{'}_1\underset{\pi_G}{\otimes}s^{'}_2)\rangle  _{\mathcal{E}_1\otimes \mathcal{E}_2}
=\langle s_1\underset{\pi_G}{\otimes}s_2,s^{'}_1\underset{\pi_G}{\otimes}s^{'}_2\rangle_{\mathcal{E}_1\rtimes G\underset{\pi_G}{\otimes} \mathcal{E}_2}$,
 $\forall s_1$, $s_1^{'}\in \mathcal{E}_1\rtimes G$ and $s_2$, $s_2^{'}\in \mathcal{E}_2$. Let us check this equality. We have:\\
$\begin{array}{lll}
\langle U(s_1\underset{\pi_G}{\otimes}s_2),U(s^{'}_1\underset{\pi_G}{\otimes}s^{'}_2)\rangle_{\mathcal{E}_1\otimes \mathcal{E}_2}&=\displaystyle \int_G\int_G \langle s_1(h)\otimes h\cdot s_2, s_1^{'}(k)\otimes k\cdot s_2^{'}\rangle_{\mathcal{E}_1\otimes \mathcal{E}_2}dhdk\\
 &=\displaystyle\int_G\int_G\langle s_1(h),s_1^{'}(k)\rangle_{\mathcal{E}_1}\langle s_2,h^{-1}k\cdot s_2^{'}\rangle_{\mathcal{E}_2} dhdk.
\end{array}$\\
By substituting $g=h^{-1}k$, we get:\\
$\begin{array}{lll}
\displaystyle \langle U(s_1\underset{\pi_G}{\otimes}s_2),U(s^{'}_1\underset{\pi_G}{\otimes}s^{'}_2)\rangle_{\mathcal{E}_1\otimes \mathcal{E}_2}&= \displaystyle\int_G\displaystyle\int_G\langle s_1(h),s^{'}_1(hg)\rangle_{\mathcal{E}_1}\langle s_2,g\cdot s_2^{'}\rangle_{\mathcal{E}_2}dhdg\\
&=\displaystyle\int_G\int_G\langle s_2,\langle s_1(h),s_1^{'}(hg)\rangle_{\mathcal{E}_1} g\cdot s_2^{'}\rangle_{\mathcal{E}_2}dhdg\\
 &=\langle s_1\underset{\pi_G}{\otimes}s_2,s^{'}_1\underset{\pi_G}{\otimes}s^{'}_2\rangle_{\mathcal{E}_1\rtimes G\underset{\pi_G}{\otimes} \mathcal{E}_2}.
\end{array}$\\
It thus remains to check that $U$ intertwines representations and operators. Let $\psi \in C^{\infty}(H)$ and $\varphi \in C^{\infty}(G)$, $s_1\in C^{\infty}(G,C^{\infty ,0}(M,E_1))$ and $s_2\in C^{\infty , 0}(M',E_2)$. We have:\\
$\begin{array}{lll}
 \pi_{H\times G}(\psi \otimes \phi)\big(U(s_1\otimes s_2)\big)\hspace*{-0.3cm}&=\displaystyle\pi_{H\times G}(\psi \otimes \varphi )\big(\int_G s_1(t)\otimes t\cdot s_2 dt\big)\\
 &=\displaystyle \int_{H\times G}\int_{G} \psi(h)\varphi(g)(h,g)\cdot (s_1(t))\otimes (gt)\cdot s_2 dtdhdg.
\end{array}$\\
Setting $gt=u$, we obtain:\\
$\begin{array}{lll}
 \pi_{H\times G}(\psi \otimes \phi)\big(U(s_1\otimes s_2)\big)\hspace*{-0.3cm}&=\displaystyle \int_{H\times G}\int_{G} \psi(h)\varphi(g)(h,g)\cdot (s_1(g^{-1}u))\otimes u\cdot s_2 dudhdg\\
 &=\displaystyle \int_{G}\int_{G} \pi_H(\psi \otimes \varphi(g))g\cdot (s_1(g^{-1}u))\otimes u\cdot s_2 dudg\\
 &=\displaystyle \int_{G}( \pi_H^{\rtimes G}(\psi \otimes \varphi) (s_1))(u)\otimes u\cdot s_2 du\\
 &=U(\pi_H^{\rtimes G}(\psi \otimes \varphi ) \underset{\pi_G}{\otimes} 1(s_1\underset{\pi_G}{\otimes} s_2)),
\end{array}$\\ 
so $U$ intertwines representations.
Let us now compute $(B\otimes 1+ 1\otimes A)(U(s_1 \underset{\pi_G}{\otimes} s_2))$. We have:
$$\begin{array}{lll}
(B\otimes 1+ 1\otimes A)(U(s_1 \underset{\pi_G}{\otimes} s_2))&=(B\otimes 1+ 1\otimes A)(\int_Gs_1(t)\otimes t\cdot s_2 dt)\\
&= \int_G B(s_1(t))\otimes t\cdot s_2 +s_1(t)\otimes t \cdot A(s_2) dt\\
&= \int_G(\tilde{B}s_1)(t)\otimes t\cdot s_2 +s_1(t)\otimes t \cdot A(s_2) dt\\
&=U((\tilde{B}\underset{\pi_G}{\otimes}1+1\underset{\pi_G}{\otimes} A)(s_1\underset{\pi_G}{\otimes} s_2)),
\end{array}$$ 
so that $U$ intertwines operators as well.\\
We now show that the operator $Q(\tilde{B}\underset{\pi_G}{\otimes}1+1\underset{\pi_G}{\otimes}A)$ is a $Q(A)$-connection on $\mathcal{E}_1\rtimes G \underset{\pi_G}{\otimes} \mathcal{E}_2$. Let us write $$Q(\tilde{B}\underset{\pi_G}{\otimes}1+1\underset{\pi_G}{\otimes}A)=M^{1/2} Q(\tilde{B})\underset{\pi_G}{\otimes} 1+N^{1/2}1\underset{\pi_G}{\otimes} Q(A),$$
where $M=\dfrac{1+\tilde{B}^2\underset{\pi_G}{\otimes}1}{1+\tilde{B}^2\underset{\pi_G}{\otimes}1+1\underset{\pi_G}{\otimes}A^2}$ and $N=\dfrac{1+1\underset{\pi_G}{\otimes}\tilde{A}^2}{1+\tilde{B}^2\underset{\pi_G}{\otimes}1+1\underset{\pi_G}{\otimes}A^2}$. Notice that the operators $M$ and $N$ are selfadjoint and bounded. Moreover $M$ and $N$ commute and we have $M+N=1+ (1+\tilde{B}^2\underset{\pi_G}{\otimes}1+1\underset{\pi_G}{\otimes}A^2)^{-1}$. To show that the operator $Q(\tilde{B}\underset{\pi_G}{\otimes}1+1\underset{\pi_G}{\otimes}A)$ is a $Q(A)$-connection on $\mathcal{E}_1\rtimes G \underset{\pi_G}{\otimes} \mathcal{E}_2$ it is sufficient to show that $M$ is a $0$-connection and that $N$ is a $1$-connexion. We obtain then by a classical result on connections that $M^{1/2}$ is a $0$-connection and that $N^{1/2}$ is a $1$-connection, so that $Q(\tilde{B}\underset{\pi_G}{\otimes}1+1\underset{\pi_G}{\otimes}A)$ is a $Q(A)$-connection since $1\underset{\pi_G}{\otimes} Q(A)$ is a $Q(A)$-connection.\\
Let us check that $M$ is a $0$-connection. Let $\varphi \in C^{\infty}(G)$ and $s\in C^{\infty,0}(M,E_1)$, we have:
$$\begin{array}{lll}
M\circ T_{s\otimes \varphi}&= \dfrac{1+B^2\underset{\pi_G}{\otimes}1}{1+B^2\underset{\pi_G}{\otimes}1+1\underset{\pi_G}{\otimes}A^2}\circ T_s \circ \pi_G(\varphi)\\
&=\dfrac{1+1\underset{\pi_G}{\otimes}A}{1+B^2\underset{\pi_G}{\otimes}1+1\underset{\pi_G}{\otimes}A^2}\circ T_{(1+B^2)s} \circ (1+A^2)^{-1}\circ \pi_G(\varphi)\\
&=N\circ T_{(1+B^2)s} \circ (1+A^2)^{-1}\circ \pi_G(\varphi).
\end{array}$$
As $(1+A^2)^{-1}\circ \pi_G(\varphi)$ is compact, we deduce that $M\circ T_{s\otimes \varphi}$ is compact because $N\in \mathcal{L}(\mathcal{E}_1\rtimes G\underset{\pi_G}{\otimes}\mathcal{E}_2)$ and $T_{(1+B^2)s}\in \mathcal{L}(\mathcal{E}_2,\mathcal{E}_1\rtimes G\underset{\pi_G}{\otimes}\mathcal{E}_2)$. Moreover, the map which associates to $s\in \mathcal{E}_1\rtimes G$ the operator $M\circ T_s$ is continuous so we obtain that $\forall s\in \mathcal{E}_1\rtimes G$, $M\circ T_s$ is compact which shows that $M$ is a $0$-connection. To show that $N$ is a $1$-connection it is sufficient to check that $(1+B^2\underset{\pi_G}{\otimes}1+1\underset{\pi_G}{\otimes}A^2)^{-1}$ is a $0$-connection since $M+N=1+ (1+B^2\underset{\pi_G}{\otimes}1+1\underset{\pi_G}{\otimes}A^2)^{-1}$. This is obtained in a way similar to the previous computation, since $(1+B^2\underset{\pi_G}{\otimes}1+1\underset{\pi_G}{\otimes}A^2)^{-1}=N\circ (1+1\underset{\pi_G}{\otimes}A^2)^{-1}$.\\ 
It remains finally to check the positivity condition. But we have:
$$[Q(B\underset{\pi_G}{\otimes} 1+ 1\underset{\pi_G}{\otimes} A),Q(B)\underset{\pi_G}{\otimes}1]=2MQ(B)^2\underset{\pi_G}{\otimes}1$$ so $\pi_H^{\rtimes G}(\theta) \underset{\pi_G}{\otimes} 1[Q(B\underset{\pi_G}{\otimes} 1+ 1\underset{\pi_G}{\otimes} A),Q(B)\underset{\pi_G}{\otimes}1]\pi_{H}^{\rtimes G}(\theta)\underset{\pi_G}{\otimes}1$ is positive modulo $\mathcal{K}(\mathcal{E}_1\rtimes G\underset{\pi_G}{\otimes}\mathcal{E}_2)$.\\
\hspace*{0.8cm} 2. For the last equality, the product of $\mathrm{Ind}^{\mathrm{M\times M'|B}} (a\sharp b)$ with $[\Delta_B]$ on the right allows to see the Hilbert $C(B\times B)$-module $\mathcal{E}_1\otimes \mathcal{E}_2$ as an Hilbert $C(B)$-module. 
\end{proof}

\subsection{Excision property}

In this section, we show an excision property which allows to compute the index of a family of $G$-transversally elliptic operators on an open subset of $M$ as the index of a family of $G$-transversally elliptic operators on the compact fibration $p : M\rightarrow B$.\\ 
We start with a lemma from \cite{Connes:Skandalis:longIndThmFoliations} which contains the ideas used in the proof of the excision theorem.

\begin{lem}$\cite{Connes:Skandalis:longIndThmFoliations}$\label{lem:support}
Let $(\mathcal{E},\pi ,F)$ be an $(A,B)$ Kasparov bimodule. Let $\mathcal{K}_1$ be the $C^{*}$-subalgebra of $\mathcal{K}(\mathcal{E})$ generated by $[\pi(a),F]$ , $\pi(a)(F^2-1)$ and $\pi(a)(F-F^*)$ for $a\in A$ and the multiples by $A$, $F$ and $F^*$. Let $\mathcal{E}_1$ be the closed $D$-submodule of $\mathcal{E}$ generated by $\mathcal{K}_1\mathcal{E}$. This module $\mathcal{E}_1$ is called the support of $(\mathcal{E},\pi , F)$, it is stable under the actions of $A$ and $F$. Let $F_1$ be the restriction of $F$ to $\mathcal{E}_1$. Then the classes of $(\mathcal{E}_1,\pi,F_1)$ and $(\mathcal{E},\pi ,F)$ coincide in $\k\k(A,B)$. 
\end{lem}

Let $G$ be a compact Lie group. Let $p : M \rightarrow B$ be $G$-equivariant compact fibration and assume that the action of $G$ on $B$ is trivial. We will use the following lemma shown in \cite[lemma 3.6]{atiyah1974elliptic} in the case where $B$ is the point, see also \cite{Atiyah-Singer:I}. The proof being identical, it is omitted. Let $U$ be a $G$-invariant open subspace of a compact fibration $p : M\rightarrow B$. Denote by $j : U \hookrightarrow M$ the inclusion. Let $\pi : T^VU \rightarrow U$ be the projection of the vertical tangent bundle defined by $T^VU = \ker ((p\circ j)_*) $, $T^VU$ is an open subspace of $T^VM$. We denote by $T^V_GU = T^VU\cap T_GM$.

\begin{lem}\label{lem:symbole:T^VU}
Each element $a\in \K(T^V_GU)$ can be represented by a homogeneous complex on $T^VU$ of degree zero
$$\xymatrix{0\ar[r] &\pi^*E^0 \ar[r]^a &\pi^*E^1 \ar[r]&0
}$$
such that, outside a compact set in $U$, the bundles $E^0$ and $E^1$ are trivial and $a$ is the identity.
\end{lem}
\noindent
We can now state our theorem.

\begin{thm}\label{thm:excision}
Let $j : U \hookrightarrow M$ be an open $G$-embedding. We assume that $p : M\rightarrow B$ is a compact $G$-fibration and that the action of $G$ is trivial on $B$. Then the composition 
$$\xymatrix{\K (T^{V}_G U) \ar[r]^{j_{*}} &\K (T^{V}_G M) \ar[rr]^{\ind }& &\k\k(C^*G,C(B))}$$
does not depend on $j$. Said differently, if $j' : U \hookrightarrow M'$ is another $G$-embedding with $M'\rightarrow B$ a compact $G$-fibration then 
$$\mathrm{Ind}^{\mathrm{M'|B}}\circ j_*'=\mathrm{Ind}^{\mathrm{M|B}} \circ j_*.$$
\end{thm}

\begin{proof}
Let $a \in \K (T^{V}_GU)$. We denote by $\pi^V : T^VU \rightarrow U$ the projection. We start by representing $a$ by a symbol of order $0$ on $T^VU$ 
$$\xymatrix{0\ar[r] &\pi^{V*}E^+ \ar[r]^{\sigma} &\pi^{V*}E^- \ar[r] &0}, $$ trivial outside a compact set $K$ of $U$, by using Lemma \ref{lem:symbole:T^VU}. We denote by $E=E^+\oplus E^-$. The inclusion $j$ defines a map $j_*$ in equivariant $\k$-theory:
$$j_* : \K(T^V_GU) \rightarrow \K(T^V_GM).$$
The map $j_*$ associates to $[\sigma ]$ the class of $\sigma$ trivially extended outside $U$ on $M$ using the trivializations. \\
Denote by $j_*E$ the vector bundle $E$ trivially extended outside $U$, and by $\mathcal{E}=\overline{C^{\infty ,0}(M,j_*E)}^{C(B)}$. We choose a family $P_0$ of pseudodifferential operators associated to $\sigma $ on $U$ such that $P_0$ is the identity outside the support $L$ of $\sigma$, that is if $s\in C^{\infty , 0}_c(U\setminus L,E^+)$ then $P_0s \in C_c^{\infty ,0}(U\setminus L,E^-)$ and it is given by
$$P_0(s)(x)=(\psi^-_x)^{-1}\circ \psi^+_x)(s(x)) ,$$
where $\psi^\pm$ are the trivializations of $E^\pm$ outside $L$. Denote by $P$ the operator $\begin{pmatrix}
0&P_0^*\\
P_0& 0
\end{pmatrix}$.
Let $\theta \in C_c^{\infty ,0}(U,[0,1])$ be a function equal to $1$ on $L$ with support $L'$. With the help of such a function $\theta$, we extend by the identity the family $P$ on $M$, and we denote by $j_*P$ this operator on $M$.  The symbol of this family $j_*P$ is $j_*\sigma $ and it is $G$-transversally elliptic along the fibers of $p : M\rightarrow B$. Let $\mathcal{E}$ be the completion of $C^{\infty ,0}(M, j^*E)$ as a $C(B)$-module defined as in section \ref{C(B):module}. We then obtain an index class $\ind(j_*P)=[\mathcal{E},\pi , j_*P]$. The map $j$ allows to see $C_c^{\infty ,0}(U,E)$ as a submodule of $\mathcal{E}$. Denote by $\mathcal{E}_{U\subset M}$ the completion of $ C_c^{\infty ,0}(U,E)$ in $\mathcal{E}$. The operator $j_*P$ preserves the closed submodule $\mathcal{E}_{U\subset M}$. Indeed, if $s\in C^{\infty ,0}_c(U,E)$ then $j_*P(s)=P(s) \in C^{\infty ,0}_c(U,E)$. Moreover, the representation $\pi$ of $C^*G$ preserves $\mathcal{E}_{U\subset M}$. Indeed, $\pi$ preserves $C_c^{\infty ,0}(U,E)$ since $U$ is a $G$-invariant open subspace in $M$, so also $\mathcal{E}_{U\subset M}$. We then get a well defined class $[\mathcal{E}_{U\subset M},\pi ,j_*P_{|_{\mathcal{E}_{U\subset M}}}]$ in $ \k\k(C^*G ,C(B))$. The cycle thus obtained is also a representative for the index class $\ind(j_*P)$ of $j_*P$. Indeed, by Lemma \ref{lem:support} above, it is sufficient to show that the two cycles have the same support, since then the operators will coincide as we will see. But if $\chi \in C^{\infty , 0}_c(U,[0,1])$ is a function which is equal to $1$ on the support of $\theta$ then $((j_*P)^2-id)(s)=((j_*P)^2-id)(\chi s)$ and hence the equality of the supports.\\
Now let $U'$ be a $G$-invariant open subspace of a compact fibration $p' : M' \rightarrow B$ which is diffeomorphic to $U$, denote by $j' : U' \hookrightarrow M'$ the inclusion. Denote by $f : U \rightarrow  U'$ a given diffeomorphism, and assume that $p\circ j=(p\circ j') \circ f$. The class of the element $a \in \K(T^V_GU)$ seen as an element of $\K(T^V_GU')$ using $f$ can be represented by $(f^{-1})^*\sigma$. We associate to $(f^{-1})^*\sigma $ the operator $(f^{-1})^*Pf^*$ that we extend by the identity on $M'$ to an operator $j^{'}_*P$. We thus get (as for $U$) an index class $\mathrm{Ind}^{\mathrm{M'|B}}(j^{'}_{*}P)=[\mathcal{E}^{'},\pi^{'} , j^{'}_{*}P]$ by taking the completion of $C^{\infty , 0}(M',j^{'}_*(f^{-1})^*E)$. We define a closed submodule $\mathcal{E}^{'}_{U'\subset M'}$ in the same way as for $\mathcal{E}_{U\subset M}$. We obtain that $\mathrm{Ind}^{\mathrm{M'|B}}(j^{'}_*P)=[\mathcal{E}^{'}_{U'\subset M'},\pi^{'} ,j^{'}_*P_{|_{\mathcal{E}_{U^\prime\subset M'}^{\prime}}}]$. But the cycle $(\mathcal{E}^{'}_{U'\subset M'},\pi^{'} ,j^{'}_*P_{|_{\mathcal{E}^{'}_{U'\subset M'}}})$ is then unitary equivalent to the cycle $(\mathcal{E}_{U\subset M},\pi ,j_*P_{|_{\mathcal{E}_{U\subset M}}})$. Indeed, if $(\mu_b)_{b\in B}$ and $(\mu^{'}_b)_{b\in B}$ are respectively the family of measures of Lebesgue type on respectively $M$ and $M'$ used in the definition of the previous index classes then we denote by $h=(h_b)_{b\in B}$ the continuous family of Radon-Nikodym derivatives associated on $U$ identified to $U'$ using $f$. We then get a unitary $V : \mathcal{E}_{U\subset M} \rightarrow \mathcal{E}_{U'\subset M'}$ given by $s  \rightarrow \sqrt{h}(f^{-1})^*s $. Furthermore, $V^{-1} (f^{-1})^*Pf^* V=\dfrac{1}{\sqrt{f^*h}}P\sqrt{f^*h}$ and $ V^{-1} \pi^{'}(\varphi ) V=\dfrac{1}{\sqrt{f^*h}}\pi(\varphi )\sqrt{f^*h}$ which give a unitarily equivalent cycle of $(\mathcal{E}_{U\subset M},\pi , j_*P_{|_{\mathcal{E}_{U\subset M}}})$. So the index class does not depend on the embedding $j$ chosen.
\end{proof}

\begin{remarque}
The equivalence between the two following cycles $(\mathcal{E},\pi ,j_*P)$ and $(\mathcal{E}_{U\subset M},\pi ,j_*P_{|_{\mathcal{E}_{U\subset M}}})$ can be deduced immediately by rewriting on the proof of Lemma \ref{lem:support}, see appendix $A$ of \cite{Connes:Skandalis:longIndThmFoliations}.
\end{remarque}

\subsection{The naturality and induction theorems}\label{NaturalityInduction}
We recall here some constructions from \cite{atiyah1974elliptic}. Let $G$ be a compact Lie group. Let $H$ be a closed subgroup of $G$. Denote by $i : H \hookrightarrow G$ the inclusion. The group $G$ is a $G\times H$-manifold for the action given by $(g,h)\cdot g'=gg'h^{-1}$  $\forall g,$ $g'\in G$ and $h\in H$. Let $p : M\rightarrow B$ be a $G$-equivariant fibration. We still assume that the action of $G$ on $B$ is trivial. We have the following product \cite{atiyah1974elliptic} (see also the previous section):
$$\k_\mathrm{H}\big(T^V_HM\big)\otimes \k_{\mathrm{G\times H}}\big(T^*_GG \big) \rightarrow \k_{\mathrm{G\times H}}\big(T^V_{G\times H}(G\times M)\big). $$
As $H$ acts freely on $G\times M$, the quotient space $Y=G\times_H M$ is a fibration on $B$ and 
$$\k_{\mathrm{G\times H}}\big(T^V_{G\times H}(G\times M) \big) \cong \K\big(T^V_GY \big).$$
The operator $0 : C^{\infty}(G) \rightarrow 0$ is $G$-transversally elliptic and $T^*_GG=G\times \{0\}$. Furthermore, $[\sigma (0)]\in \k_{\mathrm{G\times H}}(T^*_GG)\cong R(H)$ is the class of the trivial bundle. 
So we can define a map $i_* : \k_\mathrm{H}(T^V_HM) \rightarrow \K(T^V_GY )$ by the composition of the product by $[\sigma (0)]$ with the isomorphism $\k_{\mathrm{G\times H}}\big(T^V_{G\times H}(G\times M) \big) \cong \K\big(T^V_GY \big)$.\\
Recall now the definition of the element $i^* \in \k\k(C^*G, C^*H)$ from \cite{julg1982induction}. Let us fix the Haar measures on $H$ and $G$. We consider on the space $C(G)$ of continuous functions on $G$ with values in $\mathbb{C}$ the right $L^1(H)$-module structure induced by the right action of  $H$ on $G$, i.e.
$$\varphi \cdot \psi (g)=\int_H \varphi(gh^{-1}) \psi(h) dh, ~\forall \varphi \in C(G)~\mathrm{and}\  \psi \in L^{1}(H)$$
and the following hermitian structure with values in $L^1(H)$:
$$\langle f_1,f_2 \rangle (h)=\int_G \bar{f_1}(g)f_2(gh)dg.$$
Denote by $J(G)$ the completion of $C(G)$ for the norm $\|f\|=\|\langle f, f\rangle\|^{1/2}_{C^*H}$ induced by the hermitian structure. Denote by $\pi_G : C^*G \rightarrow \mathcal{L}_{C^*H}(J(G))$ the natural representation of $C^*G$ induced by the left $G$ action on itself.

\begin{defi}$\cite{julg1982induction}$
The element $i^* \in \k\k(C^*G,C^*H)$ is defined as the class of the cycle $(J(G),\pi_G ,0).$
\end{defi}

\begin{lem}
The element $i^*$ is the image of the trivial representation of $H$ seen as a trivial $G$-equivariant vector bundle on $G/H$ by the Morita equivalence between $C(G/H)\rtimes G$ and $C^*H$. 
\end{lem}

\begin{proof}
In $\k\K(\mathbb{C},C(G/H))$ the trivial representation of $H$ is given by the cycle $(C(G/H),0)$. Its image by the Kasparov descent $j^G([C(G/H),0])=[C(G/H)\rtimes G ,\pi_G , 0]$ is an element of $\k\k(C^*G,C(G/H)\rtimes G)$. The cycle given by the Morita equivalence is $[J(G) , \rho , 0]$, where $\rho : C(G/H)\rtimes G\rightarrow \mathcal{L}_{C^*H}(J(G))$ is given by $\rho (\theta )s(g)=\int_G \theta(t,g)s(t^{-1}g) dt $ for all $\theta \in C(G,C(G/H))$ and $s\in C(G)$. 
\end{proof}

We also denote by $i_* : \k\k(C^*H ,C(B)) \rightarrow \k\k(C^*G,C(B))$ the left Kasparov product by $i^*\in \k\k(C^*H ,C^*G)$.
\begin{thm}$\cite{atiyah1974elliptic}$\label{thm:induction:1}
The following diagram is commutative.

$$\xymatrix{\k_\mathrm{H}\big(T^V_HM \big) \ar[r]^{i_*}\ar[d]_{\ind } & \K\big( T^V_GY \big ) \ar[d]^{\mathrm{Ind}^{Y|B}}\\
\k\k(C^*H,C(B)) \ar[r]_{i_*} &\k\k(C^*G ,C(B))
}$$

\end{thm} 

\begin{proof}
Let $a\in \k_\mathrm{H}(T^V_HM )$. From the multiplicative property of the index, we have:
$$\mathrm{Ind^{G\times M|B}}(a\cdot [\sigma (0)]) =j^H\big(\mathrm{Ind}^{\mathrm{G|\star} }([\sigma (0)]) \big)\otimes_{C^*H} \ind(a).$$
Furthermore, the isomorphism $\k_{\mathrm{G\times H}}\big(T^V_{G\times H}(G\times M)\big) \cong \K(T^V_GY)$ send $a\cdot [\sigma(0)]$ on $i_*(a)$ by definition of $i_*$. By Theorem \ref{thm:action:libre}, we know that $\mathrm{Ind}^{\mathrm{Y|B}}(i_*(a))=\chi_0^H\otimes_{C^*H}\mathrm{Ind}^{\mathrm{G\times M|B}}(a\cdot [\sigma(0)])$. From this we deduce that: 
$$\mathrm{Ind}^{\mathrm{Y|B}}(i_*(a))=\chi_0^H\otimes_{C^*H}\big[j^H\big(\mathrm{Ind}^{\mathrm{G|\star} }([\sigma (0)]) \big)\otimes_{C^*H} \ind(a)\big].$$
It is easy to see that $i_*=\chi_0^H\otimes_{C^*H}j^H\big(\mathrm{Ind}^{\mathrm{G|\star} }([\sigma (0)]) \big)$ coincides with $i^*$. Indeed, the action of $H$ on $C^*G$ is trivial here so this follows from a trivial generalization of the Green-Julg theorem, see \cite{julg1981produit:croise}. We get the result by associativity of the Kasparov product.
\end{proof}

\begin{remarque}
The previous theorem allows for instance to reduce the index problem to the case of connected Lie groups, since any compact Lie group can be embedded in a unitary group.
\end{remarque}

In order to reduce the computation of the index map to
the case of torus action, we need to investigate the induction map. Let us first note that if $M$ is a compact $G$-fibration then the map $(g,m) \rightarrow (gH,g\cdot m)$ is a $G$-equivariant diffeomorphism between the $G$-fibrations $Y=G\times_HM$ and $G/H\times M$ over $B$. Let us assume that $G$ is connected and that $H$ is a maximal torus so that the homogeneous space $G/H$ is a complex manifold, and hence is $\k$-oriented. Then we have a Dolbeault operator $\overline{\partial}$ on $G/H$ whose $G$-equivariant index is $\mathrm{Ind}(\overline{\partial})=1 \in R(G)$. By the multiplicative property of the index, we have the following commutative diagram:
$$\xymatrix{\K(T^V_GM) \ar[d]_{\mathrm{Ind}^{\mathrm{M|B}}}& \hspace*{-1cm}\bigotimes &\hspace*{-1cm} \K(T^*(G/H))\ar[r]\ar[d]^{\mathrm{Ind}^{\mathrm{G/H|\star}}}&\K(T^V_GY)\ar[d]^{\mathrm{Ind}^{\mathrm{Y|B}}}\\
\k\k(C^*G,C(B))& \hspace*{-1cm}\bigotimes & \hspace*{-1cm}\k\K(\mathbb{C},\mathbb{C}) \ar[r] &\k\k(C^*G,C(B)).
}$$ 
Multiplication by the symbol $[\sigma (\overline{\partial})]\in \K(T^*(G/H))$ induces a morphism $$k : \K(T^V_GM) \rightarrow \K(T^V_GY)$$ which preserves the index map. More precisely: 

\begin{thm}
Let $H$ be a maximal torus of a connected compact Lie group $G$. Denote by $r= (i_*)^{-1}\circ k : \K(T^V_GM)\rightarrow \k_{\mathrm{H}}(T^V_HM)$ the composition of the morphism $k$ with the inverse of $i_*$. The following diagram then commutes: 
$$\xymatrix{\K(T^V_GM) \ar[r]^r \ar[d]_{\mathrm{Ind}^{M|B}}& \k_H(T^V_HM)\ar[d]^{\mathrm{Ind}^{M|B}}\\
\k\k(C^*G,C(B)) &\k\k(C^*H,C(B)) \ar[l]^{i_*}.
}$$
\end{thm}

\begin{proof}The following diagram is commutative:
$$\xymatrix{\K(T^V_GM) \ar[r]^k \ar[d]_{\mathrm{Ind}^{\mathrm{M|B}}}& \K(T^V_GY)\ar[d]^{\mathrm{Ind}^{\mathrm{Y|B}}}\\
\k\k(C^*G,C(B)) \ar[r]_{=} &\k\k(C^*G,C(B)) .
}$$
Indeed, if $a\in \K(T^{V}_GM)$, we have:
$$\begin{array}{lll}
\mathrm{Ind}^{\mathrm{Y|B}}(k(a))&=\mathrm{Ind}^{\mathrm{Y|B}}(a\cdot [\sigma(\overline{\partial})])\\
&=j^G\big(\mathrm{Ind}(\overline{\partial})\big)\otimes_{C^*G} \ind(a)\\
&=j^G\big(1\big)\otimes_{C^*G} \ind(a)\\
&=1\otimes_{C^*G} \ind(a)\\
&=\ind(a).
\end{array}$$
By Theorem \ref{thm:induction:1}, we get that the following diagram is commutative:
$$\xymatrix{\K(T^V_GM) \ar[r]^k \ar[d]_{\mathrm{Ind}^{\mathrm{M|B}}}& \K(T^V_GY)\ar[d]^{\mathrm{Ind}^{\mathrm{Y|B}}}& \ar[l]_{i_*} \k_\mathrm{H}(T^V_HM)\ar[d]^{\mathrm{Ind}^{\mathrm{M|B}}}\\
\k\k(C^*G,C(B)) \ar[r]_{=} &\k\k(C^*G,C(B)) &\k\k(C^*H,C(B)) \ar[l]^{i_*}.
}$$

\end{proof}

We end this section with the statement of Theorem \ref{thm:naturalité:!} below which shows the compatibility of the index map with the Gysin map.
Let $j : M \hookrightarrow M'$ be an inclusion of $G$-fibrations on $B$ such that $p'\circ j=p$ where $p : M \rightarrow B$ and $p' : M'\rightarrow B $ are the projections. Let us assume that $M$ is compact. Following \cite{atiyah1974elliptic}, we will define a morphism of $R(G)$-modules $j_! : \K(T^V_GM) \rightarrow \K(T^V_GM')$.\\
Let $N$ be a $G$-invariant tubular neighbourhood of $M$ in $M'$. Then $N$ is a $G$-manifold and it can be identified with the normal bundle of $M$ in $M'$. The vector bundle $T^VM$ is a closed $G$-submanifold of $T^VM'$ and the tubular neighbourhood $U$ of $T^VM$ in $T^VM'$ can be identified with the vector bundle given by lifting the vector bundle $N\oplus N$ by the projection $\pi :T^VM \rightarrow M$. The exterior algebra of $\pi^*(N\otimes \mathbb{C})$ defines a complex on $T^VN$ which is exact outside the zero section $T^VM \subset T^VN$. We denote it by $\bigwedge(T^VN)$. Denote by $q : T_G^VN \rightarrow T^V_GM$ the projection. If $E$ is a complex on $T^V_GM$ with compact support then the product $\bigwedge(T^VN)_{|T_GN}\otimes q^*E$ has compact support in $T_G^VN$. The product by $\bigwedge (T^VN)$ defines a morphism of $R(G)$-modules 
$$\phi : \K(T^V_GM)\rightarrow \K(T_G^VN),$$
called the Thom homomorphism. As $T_G^VN$ is an open subspace of $T^V_GM'$, the inclusion induces a map 
 $$k_* : \K(T_G^VN) \rightarrow \K(T^V_GM').$$  
The composition of these two maps gives the desired Gysin map
$$j_! : \K(T^V_GM) \rightarrow \K(T^V_GM').$$

\begin{thm}$\cite{atiyah1974elliptic}$\label{thm:naturalité:!}
Let $j :M\hookrightarrow M'$ be a $G$-embedding over $B$ with $M$ compact.
The following diagram is commutative:
$$\xymatrix{\K(T^V_GM) \ar[r]^{j_!} \ar[d]_{\ind}& \K(T^V_GM') \ar[d]^{\mathrm{Ind}^{\mathrm{M'|B}}} \\
\k\k(C^*G ,C(B)) \ar@{=}[r] &\k\k(C^*G ,C(B)).
}$$
%where $\mathrm{Ind}^{\mathrm{M'|B}}$ is well defined thanks to the excision theorem \ref{thm:excision}.
\end{thm}

\begin{proof}
By the excision theorem, the index commutes with $k_*$. It is thus sufficient to show the theorem in the case of a real $G$-vector bundle $N$ on $M$ and with $j_! : \K(T_G^VM) \rightarrow \K(T_G^VN)$ being the Thom homomorphism. We can write $N = P\times_{O(n)} \mathbb{R}^n$, where $P$ is a $O(n)$-principal bundle on $M$. The $G$-action on $P$ commutes with the $O(n)$-action and the $G$-action is trivial on $\mathbb{R}^n$. We have the following product: 
$$\k_{\mathrm{G\times O(n)}}(T_{G\times O(n)}^VP) \otimes \k_{\mathrm{G\times O(n)}}(T\mathbb{R}^n)\rightarrow \k_{\mathrm{G\times O(n)}}(T_{G\times O(n)}^V(P\times \mathbb{R}^n)),$$
and the following isomorphisms:
$$\begin{array}{lll}
q_1^* : \K(T^V_GM) &\rightarrow \k_{\mathrm{G\times O(n)}}(T^V_{G\times O(n)}P)\\
q_2^* : \K(T^V_GN) &\rightarrow \k_{\mathrm{G\times O(n)}}(T_{G\times O(n)}^V(P \times \mathbb{R}^n)).
\end{array}$$
Therefore, we obtain a product: 
$$\K(T_{G}^VM) \otimes \k_{\mathrm{G\times O(n)}}(T\mathbb{R}^n)\rightarrow \K(T^V_GN).$$
The inclusion of the origin in $\mathbb{R}^n$ induces a Bott morphism $i_! : R(O(n))\rightarrow \k_{\mathrm{O(n)}}(T\mathbb{R}^n)$ and we have $\mathrm{Ind}(i_!(1))=1 \in \k\k_{O(n)}(\mathbb{C},\mathbb{C})$ by \cite{Atiyah-Singer:I}. As $G$ acts trivially on $\mathbb{R}^n$, we get that $\mathrm{Ind}(i_!(1))=1 \in \k\k_{\mathrm{G\times O(n)}}(\mathbb{C},\mathbb{C})$. The product by $i_!(1)$ makes the following diagram commutative 
$$\xymatrix{ \K(T^V_GM) \ar[r]^{j_!} \ar[d]_{q_1^*}& \K(T^V_GN)\ar[d]^{q_2^*} \\
\k_{\mathrm{G\times O(n)}}(T^V_{G\times O(n)}P)\ar[r]_{\hspace*{-0.5cm}\cdot i_!(1)} & \k_{\mathrm{G\times O(n)}}(T^V_{G\times O(n)}(P\times \mathbb{R}^n)) 
}$$
from this it follows that 
$$\mathrm{Ind}^{\mathrm{P\times \mathbb{R}^n|B}}(q_2^*(j_!(a))=\mathrm{Ind}^{\mathrm{P\times \mathbb{R}^n|B}}(q_1^*(a)\cdot i_!(1)),\ \forall a\in\K(T^V_GM).$$
By the multiplicative property of the index, we obtain:
$$\mathrm{Ind}^{\mathrm{P\times \mathbb{R}^n|B}}(q_2^*(j_!(a))=j^{G\times O(n)}\big(1\big) \otimes_{C^*G\otimes C^*O(n)}\mathrm{Ind}^{\mathrm{P|B}}(q_1^*(a)).$$ 
Furthermore, we have the following equalities by Theorem \ref{thm:action:libre}:
$$\ind (a )=\chi_0^{O(n)} \otimes_{C^*O(n)} \mathrm{Ind}^{\mathrm{P|B}}(q_1^*(a)),$$
$$\mathrm{Ind}^{\mathrm{N|B}}(j_!(a))=\chi_0^{O(n)}\otimes_{C^*O(n)} \mathrm{Ind}^{\mathrm{P\times \mathbb{R}^n|B}}(q_2^*(j_!(a)),$$
where $\chi_0^{O(n)}$ is the $O(n)$ trivial representation.
Finally we conclude that $\ind (a )=\mathrm{Ind}^{\mathrm{N|B}}(j_!(a))$ since $j^{G\times O(n)}(1)=1$.

\end{proof}

\begin{remarque}
The properties of the index proved above allow to reduce the computation of the index map to the case of a fibration $B\times \mathbb{R}^n \rightarrow B $ equipped with a torus action on $\mathbb{R}^n$. 
\end{remarque}

\section{Kasparov's intersection product of the index class with an elliptic operator on the base}

In the present section, we insist that $B$ is a manifold. This was not rigorously needed in the previous sections.
We compute now the Kasparov product of the index class of family of $G$-transversally elliptic operators with the $\k$-homology class of an elliptic operator on the base $B$.
We begin by recalling some definitions and results from \cite{hilsum1987morphismes}. \\

\begin{prop}$\cite[\mathrm{Proposition~A.8}]{hilsum1987morphismes}$\label{prop:Q-connexion}
Let $E_1$ be a hermitian bundle on $M$ and $E_2$ be a hermitian bundle on $B$.
Let $Q : C^{\infty}(B,E_2) \rightarrow C^{\infty}(B,E_2)$ be a pseudodifferential operator on $B$ of order $0$. Let $Q' : C^{\infty}(M,E_1\otimes p^*E_2) \rightarrow C^{\infty}(M,E_1\otimes p^*E_2)$ be a pseudodifferential operator of order $0$ such that
$$\sigma_{|p^*T^*B}(Q')=1_{E_1}\hat{\otimes}p^*\sigma(Q) ,$$
where $p^*\sigma(Q)(m,\xi)=\sigma(Q)(p(m),\xi)$, $\forall (m,\xi)\in p^*T^*B$.
Then $Q'$ is a $Q$-connection for $\mathcal{E}_1$, where $\mathcal{E}_1$ is the $C(B)$-module associated to $E_1$ defined in Section \ref{C(B):module}.
\end{prop}

\begin{prop}$\mathrm{\cite[\mathrm{proposition~A.10.2}]{hilsum1987morphismes}}$\label{prop:crochet:pseudodiff+positivité}
Let $P : C^{\infty,0}(M,E_1)\rightarrow C^{\infty,0}(M,E_1)$ be a family of pseudodifferential operators of order $0$ on $M$. Let $Q'$ be a pseudodifferential operator of order $0$ on $M$ such that
$$\sigma_{|p^*T^*B}(Q')=1_{E_1}\hat{\otimes} \sigma ',$$
where $\sigma'$ is a symbol of order $0$ on $B$. Then $[P\hat{\otimes} 1, Q'] : C^{\infty}(M,E_1\hat{\otimes}p^*E_2)\rightarrow C^{\infty}(M,E_1\hat{\otimes}p^*E_2) $ is a pseudodifferential operator of order $0$ and $$\sigma ([P\hat{\otimes} 1 ,Q'])=[\sigma(P)\hat{\otimes}1,\sigma (Q')].$$ 
\end{prop}

\begin{remarque}
In the previous two propositions, all operators and all symbols can obviously be taken $G$-invariant.
\end{remarque}

Let us check that the product, between $G$-transversally elliptic symbols along the fibers and elliptic symbols on the base, gives $G$-transversally elliptic symbols on $M$.

\begin{lem}
The map $$\K(T^{*V}_GM)\otimes \k(T^*B)\rightarrow \K(T^*_GM)$$
defined by 
$$(\sigma\hat{\otimes}1+1\hat{\otimes}p^*\sigma')(m,\xi,\xi^{'})=\sigma(m,\xi)\hat{\otimes}1+1\hat{\otimes}\sigma'(p(m),\xi^{'}),$$
for $\sigma \in \K(T_G^{*V}M)$, $\sigma'\in \k(T^*B)$, $\xi\in T_G^VM_m$ and $\xi^{'} \in T_{p(m)}^*B$ induces a well defined product in $\k$-theory.
\end{lem}

\begin{proof}
On the one hand, the map $\K(T^{*V}_GM)\otimes \k(T^*B)\rightarrow \K(T^{*V}_GM\times p^*T^*B)$ defined by $(\sigma ,\sigma' ) \mapsto \sigma \hat{\otimes}1+1\hat{\otimes}p^*\sigma'$ induces a well defined product. On the other hand, $T^*_GM=T^{*V}_GM\times_M p^*T^*B$ so by restriction of the product in $ T^{*V}_GM\times p^*T^*B$ to $T^{*V}_GM\times_M p^*T^*B$, we obtain a product in $\k$-theory with values in $\K(T^*_GM)$.

\end{proof}

\begin{thm}\label{thm:pair:ell}
Let $P : C^{\infty,0}(M,E_1)\rightarrow C^{\infty,0}(M,E_1)$ be a $G$-invariant selfadjoint family of $G$-transversally elliptic pseudodifferential operators of order $0$. Let $Q : C^{\infty}(B,E_2)$ be an elliptic pseudodifferential operator on $B$ of order $0$. The Kasparov product of the class $[\mathcal{E}_1,\pi,P]=\ind(P)$ by $[Q]\in \k\k(C(B),\mathbb{C})$ is given by $[\mathcal{E}_1\hat{\otimes}_{C(B)}\mathcal{E}_2 , Q']$, where $Q'$ is any $G$-invariant pseudodifferential operator of order $0$ which is $G$-transversally elliptic on $M$ with symbol given by 
$$\sigma (Q')=\sigma (P)\hat{\otimes }1+1\hat{\otimes}p^*\sigma (Q).$$

\end{thm}

\begin{proof}
Notice first that if $Q'$ is a $G$-invariant and $G$-transversally elliptic operator on $M$ then $\mathrm{Ind}^M(Q') \in \k\k(C^*G,\mathbb{C})$ is well defned and only depends on its symbol $[\sigma(Q')]$, see \cite{atiyah1974elliptic}. Let us check that such $Q'$ is a $Q$-connection for $\mathcal{E}_1$ and that $\forall \varphi \in C^*G$, $\pi(\varphi)[P\hat{\otimes}1, Q']\pi(\varphi)^*$ is positive modulo compact operators.  Since $ \sigma (Q')_{|p^*T^*B}=1\hat{\otimes}p^*\sigma(Q)$, we get by Proposition \ref{prop:Q-connexion} that $Q'$ is a $Q$-connection for $\mathcal{E}_1$. On the other hand, we obtain by Proposition \ref{prop:crochet:pseudodiff+positivité} that $[P\hat{\otimes}1, Q']$ is a pseudodifferential operator on $M$ and that its symbol is given by
$$\sigma ([P\hat{\otimes} 1 ,Q'])=[\sigma(P)\hat{\otimes}1,\sigma (Q')]=\sigma(P)^2\hat{\otimes}1,$$
which is positive. Moreover, $Q'$ and $P\hat{\otimes} 1$ are $G$-invariant so they commute with $\pi(\varphi)$. It remains to show that $[P\hat{\otimes} 1 ,Q']$ is positive modulo compact operators. But $\sigma ([P\hat{\otimes} 1 ,Q'])\geq 0$ so 
$[P\hat{\otimes} 1 ,Q']\geq 0$ modulo compact operators.
\end{proof}

\begin{cor}
The pairing of the index class of a $G$-invariant family $P$ of $G$-transversally elliptic pseudodifferential operators of order $0$ with an element of the $\k$-homology $\k\k(C(B),\mathbb{C})$ of $B$ is given by the index class of a $G$-invariant pseudodifferential operator which is $G$-transversally elliptic on $M$.  
\end{cor}

\begin{proof}
As any element of $\k\k(C(B),\mathbb{C})$ can be represented by an elliptic operator on $B$, this is a consequence of Theorem \ref{thm:pair:ell}.
\end{proof}

\begin{cor}\label{cor:couplage:ind:Atiyah}
If $\alpha $ is an element of the $\k$-homology group of $B$ then the class $[\mathcal{E},\pi, P]\otimes_{C(B)} \alpha \in \k\k(C^*G,\mathbb{C})\simeq \mathrm{Hom}(R(G),\mathbb{C})$ is given by the distributional index of Atiyah \cite{atiyah1974elliptic}, i.e. the multiplicities of $ \ind(P)\otimes_{C(B)} \alpha $ are summable in the sense of distributions on $G$. Denote by $m([V]\otimes_{C*G}\ind(P)\otimes_{C(B)}\alpha)$ the integer associated to the multiplicity $[V]\otimes_{C*G}\ind(P)\otimes_{C(B)}\alpha$ of $V$ in $\ind(P)\otimes_{C(B)} \alpha $. We have that for any $ \varphi \in C^\infty(G)$:\\
the sum
$$\sum \limits_{V\in \hat{G}}m \big([V]\otimes_{C*G}\ind(P)\otimes_{C(B)}\alpha \big)\langle\chi_V,\varphi\rangle_{L^2(G)}$$
is convergent. 
\end{cor}

\begin{proof}
Indeed, by the previous proposition $\ind(P)\otimes_{C(B)}\alpha $ is represented by the index class of a $G$-transversally elliptic pseudodifferential operator on $M$. By \cite{atiyah1974elliptic}, we know that the distributional index of a $G$-transversally elliptic operator $Q$ is tempered on $G$ and that it is totally determined by its multiplicities, that is to say $\mathrm{Ind}^M(Q)=\displaystyle \sum\limits_{V\in \hat{G}}m_Q(V)\chi_V$ is defined as distribution.
\end{proof}

\appendix
\section{Some results on families of pseudodifferential operators}\label{Appendix}

\noindent
In this appendix, we recall some classical results on families of pseudodifferential operators. We refer the reader to \cite{lauter2000pseudodiff,lauter2005spectral,lauter1999pseudodiff,monthubert1997indice,nistor1999pseudodiff,Paterson07theequivariant,baldare:these}.

 For the following two well known results see \cite[Theorem 2.2.1 and Proposition 2.2.2.]{Hormander1971}  (see also \cite[Theorem 6.2 and Proposition 6.1]{shubin2001pseudodifferential}  and \cite[Theorem 3.2 and Proposition 3.3]{Kasparov:KKindex}) for the standard results.

\begin{thm}\cite[Theorem 3.2]{Kasparov:KKindex}\label{thm:limitesupfam}
Let P be a $G$-invariant family of pseudodifferrntial operators of order $0$. Assume that its principal symbol $\sigma_P$ is bounded at infinity in the vertical cotangent direction by a constant $C>0$, i.e. $\forall K\subset M$ compact
$$\varlimsup \limits_{\substack{ \xi \to \infty \\ \xi \in T^VM^*}}\|\sigma_P(x,\xi)\|:=\lim \limits_{t\to \infty} \sup \limits_{\substack{|\xi|\geq t\\x\in K}} \|\sigma_P(x,\xi)\|<C. $$
Then there exist a $G$-invariant family of pseudodifferential operators $S$ of order $0$ and a $G$-invariant family of integral operators $R$ with continuous kernel such that $P^*P+S^*S-C^2=R$.
\end{thm}

\begin{prop}\cite[Proposition 3.3]{Kasparov:KKindex}\label{prop:limiteinffam}
Let $Q$ be a $G$-invariant family of pseudodifferential operators of order $0$ such that $Q^*=Q$. Assume that its principal symbol $\sigma_Q$ satisfies that for any compact subset $K\subset M$:
$$\varliminf \limits_{\substack{\xi \to \infty \\ \xi\in T^VM^*  }}\mathrm{Re}\big( \sigma_Q(x,\xi)\big):=\lim \limits_{t\to 0}\inf \limits_{\substack{|\xi| \geq t\\x\in K}} \mathrm{Re}\big( \sigma_Q(x,\xi)\big) >0.$$
Then there exists a $G$-invariant family of operators $S$ of order $0$ such that
$$R=S^*S-Q$$
have continuous kernel. 

\end{prop}

\begin{lem}Let $R$ be a continuous family of integral operators with continuous kernels.
\begin{enumerate}
\item The family $R$ extends to a bounded operator on $\mathcal{E}$.
\item We have $R\in \mathcal{K}(\mathcal{E})$.
\end{enumerate}
\end{lem}

\begin{prop}\label{prop:Pborné} Let $P_0 : C^{\infty ,0 }(M,E^+) \rightarrow C^{\infty ,0}(M,E^-)$ be a family of $G$-invariant pseudodifferential operators of order $0$.
% which is transversally elliptic along the fibers. 
We denote by $P : C^{\infty ,0}(M,E) \rightarrow C^{\infty ,0}(M,E)$, the family of operators $\begin{pmatrix}
0&P_0^*\\
P_0&0
\end{pmatrix}$.
The family $P$ extends to a $G$-invariant bounded operator of $\mathcal{L}_{C(B)}(\mathcal{E})$.

\end{prop}

\begin{thm}\label{thm:compact}
Let $P$ be a family of pseudodifferential operators of order $0$. Assume that $$\varlimsup \limits_{\substack{ \xi \to \infty \\ \xi \in T^VM^*}} \|\sigma_P(m,\xi)\|=0$$ then $P\in \mathcal{K}(\mathcal{E})$.
\end{thm}

\begin{cor}\label{cor:compact}
Let $P$ be a family of negative order. Then $P$ extends to a compact operator on $\mathcal{E}.$
\end{cor}

\footnotesize
\bibliographystyle{plain}
\bibliography{Transversalement_elliptique}

\begin{thebibliography}{10}

\bibitem{atiyah1974elliptic}
M.F. Atiyah.
\newblock {\em Elliptic operators and compact groups}.
\newblock Lecture notes in mathematics. Springer Verlag, 1974.

\bibitem{Atiyah-Singer:I}
M.F. Atiyah and I.M. Singer.
\newblock The index of elliptic operators {I}.
\newblock {\em Ann. Math.}, 87:484--530, 1968.

\bibitem{Atiyah-Singer:IV}
M.F. Atiyah and I.M. Singer.
\newblock The index of elliptic operators {IV}.
\newblock {\em Ann. Math.}, 93:119--138, 1971.

\bibitem{baaj1983theorie}
S.~Baaj and P.~Julg.
\newblock Th{\'e}orie bivariante de {K}asparov et op{\'e}rateurs non born{\'e}s
  dans les {C}*-modules hilbertiens.
\newblock {\em CR Acad. Sci. Paris S{\'e}r. I Math}, 296(21):875--878, 1983.

\bibitem{baldare:cohomologie}
A.~Baldare.
\newblock The index of {G}-transversally elliptic families {II}.
\newblock 2018.

\bibitem{baldare:these}
A.~Baldare.
\newblock {\em Theorie de l'indice pour les familles d'operateurs
  {G}-transversalement elliptiques}.
\newblock PhD thesis, Université de Montpellier, 2018.

\bibitem{Benameur:LongLefschetzKtheorie}
M.-T. Benameur.
\newblock A longitudinal {L}efschetz theorem in {K}-theory.
\newblock {\em K-theory}, 12:227--257, 1997.

\bibitem{Benameur:thmFamilleLefschetz}
M.-T. Benameur.
\newblock Cyclic cohomology and the family {L}efschetz theorem.
\newblock {\em Math. Ann.}, 323:97–121, 2002.

\bibitem{Benameur:flat:bundles}
Moulay-Tahar Benameur.
\newblock A higher lefschetz formula for flat bundles.
\newblock {\em Transactions of the American Mathematical Society},
  355(1):119--142, 2003.

\bibitem{BV:ChernCharacterTransversally}
N.~Berline and M.~Vergne.
\newblock The {C}hern character of a transversally elliptic symbol and the
  equivariant index.
\newblock {\em Inventiones mathematicae}, 124(1):11--49, 1996.

\bibitem{BV:IndEquiTransversal}
N.~Berline and M.~Vergne.
\newblock L'indice {\'e}quivariant des op{\'e}rateurs transversalement
  elliptiques.
\newblock {\em Inventiones mathematicae}, 124(1):51--101, 1996.

\bibitem{Connes:Skandalis:longIndThmFoliations}
A.~Connes and G.~Skandalis.
\newblock The longitudinal index theorem for foliations.
\newblock {\em Publications RIMS Kyoto Univ}, 20:135--179, 1984.

\bibitem{dixmier1963champs}
J.~Dixmier and A.~Douady.
\newblock Champs continus d'espaces hilbertiens et de {C}*-alg{\'e}bres.
\newblock {\em Bulletin de la Soci{\'e}t{\'e} math{\'e}matique de France},
  91:227--284, 1963.

\bibitem{fox1994index}
J.~Fox and P.~Haskell.
\newblock The index of transversally elliptic operators for locally free
  actions.
\newblock {\em Pacific Journal of Mathematics}, 164(1):41--85, 1994.

\bibitem{hilsum1989fonctorialite}
M.~Hilsum.
\newblock Fonctorialit{\'e} en {K}-th{\'e}orie bivariante pour les
  vari{\'e}t{\'e}s lipschitziennes.
\newblock {\em K-theory}, 3(5):401--440, 1989.

\bibitem{hilsum2010bordism}
M.~Hilsum.
\newblock Bordism invariance in {KK}-theory.
\newblock {\em Mathematica Scandinavica}, 107(1):73--89, 2010.

\bibitem{hilsum1987morphismes}
M.~Hilsum and G.~Skandalis.
\newblock Morphismes {K}-orient{\'e}s d'espaces de feuilles et
  fonctorialit{\'e} en th{\'e}orie de {K}asparov (d'apr{\`e}s une conjecture
  d'{A}. {C}onnes).
\newblock 20(3):325--390, 1987.

\bibitem{Hormander1971}
L.~H{\"o}rmander.
\newblock Fourier integral operators. {I}.
\newblock {\em Acta Mathematica}, 127(1):79--183, 1971.

\bibitem{julg1981produit:croise}
P.~Julg.
\newblock K-th{\'e}orie {\'e}quivariante et produits crois{\'e}s.
\newblock {\em CR Acad. Sci. Paris S{\'e}r. I Math}, 292:629--632, 1981.

\bibitem{julg1982induction}
P.~Julg.
\newblock Induction holomorphe pour le produit crois{\'e} d’une {C}*-algebre
  par un groupe de {L}ie compact.
\newblock {\em CR Acad. Sci. Paris S{\'e}r. I Math}, 294(5):193--196, 1982.

\bibitem{julg1988indice}
P~Julg.
\newblock Indice relatif et {K}-th{\'e}orie bivariante de {K}asparov.
\newblock {\em CR Acad. Sci. Paris}, 307:243--248, 1988.

\bibitem{Kasparov:KKtheory}
G.~Kasparov.
\newblock The operator {K}-functor and extensions of c*-algebras.
\newblock {\em Mathematics of the USSR-Izvestiya}, 16(3):513, 1981.

\bibitem{Kasparov1988}
G.~Kasparov.
\newblock Equivariant {KK}-theory and the {N}ovikov conjecture.
\newblock {\em Inventiones mathematicae}, 91(1):147--202, 1988.

\bibitem{Kasparov:KKindex}
G.~Kasparov.
\newblock Elliptic and transversally elliptic index theory from the viewpoint
  of ${KK}$-theory.
\newblock {\em Journal of Noncommutative Geometry}, 10(4):1303--1378, 2016.

\bibitem{lance1995hilbert}
E.C. Lance.
\newblock {\em Hilbert {C}*-{M}odules: {A} {T}oolkit for {O}perator
  {A}lgebraists}.
\newblock Lecture note series / London mathematical society. Cambridge
  University Press, 1995.

\bibitem{lauter2000pseudodiff}
R.~Lauter, B.~Monthubert, and V.~Nistor.
\newblock Pseudodifferential analysis on continuous family groupoids.
\newblock {\em Doc. Math}, 5:625--655, 2000.

\bibitem{lauter2005spectral}
R.~Lauter, B.~Monthubert, and V.~Nistor.
\newblock Spectral invariance for certain algebras of pseudodifferential
  operators.
\newblock {\em Journal of the Institute of Mathematics of Jussieu},
  4(3):405--442, 2005.

\bibitem{lauter1999pseudodiff}
R.~Lauter and V.~Nistor.
\newblock {\em Pseudodifferential analysis on groupoids and singular spaces}.
\newblock Johannes-Gutenberg-Univ., Fachbereich Mathematik, 1999.

\bibitem{monthubert1997indice}
B.~Monthubert and F.~Pierrot.
\newblock Indice analytique et groupo{\"\i}des de lie.
\newblock {\em Comptes {R}endus de l'{A}cad{\'e}mie des {S}ciences-{S}eries
  {I}-{M}athematics}, 325(2):193--198, 1997.

\bibitem{nistor1999pseudodiff}
V.~Nistor, A.~Weinstein, and P.~Xu.
\newblock Pseudodifferential operators on differential groupoids.
\newblock {\em Pacific journal of mathematics}, 189(1):117--152, 1999.

\bibitem{Paterson07theequivariant}
A.L.T. Paterson.
\newblock The equivariant analytic index for proper groupoid actions, 2007.

\bibitem{shubin2001pseudodifferential}
M.A. Shubin.
\newblock {\em Pseudodifferential {O}perators and {S}pectral {T}heory}.
\newblock Pseudodifferential Operators and Spectral Theory. Springer Berlin
  Heidelberg, 2001.

\bibitem{C*algebre:C*module}
G.~Skandalis.
\newblock {\em C*-alg{\'e}bre, alg{\'e}bre de von {N}eumann, exemples}.

\end{thebibliography}

\end{document}